\documentclass[reqno, 12pt]{article}

\usepackage{enumerate}
\usepackage{needspace}
\usepackage{latexsym}
\usepackage{amsmath}
\usepackage{amsfonts}
\usepackage{amsthm}
\usepackage{amssymb,mathtools}
\usepackage{newlfont}
\usepackage{graphics}
\usepackage{color}
\usepackage{float}
\usepackage{diagbox}
\usepackage{tocloft}
\usepackage{titlesec}
\usepackage{booktabs,array,longtable,enumitem,microtype}
\usepackage{extpfeil}
\usepackage{centernot}

\usepackage[linesnumbered,ruled,vlined]{algorithm2e}
\usepackage{url}

\usepackage[margin=1in]{geometry}

\usepackage[T1]{fontenc}
\usepackage{lmodern}
\usepackage{microtype}
\usepackage{rotating}
\usepackage{multirow}
\usepackage{extarrows}
\usepackage[sort,compress,numbers]{natbib}
\usepackage[utf8]{inputenc}
\usepackage{xcolor}
\usepackage{listings}
\usepackage{aliascnt}
\usepackage[pagebackref,colorlinks=true,linkcolor=blue,citecolor=red,urlcolor=blue]{hyperref}
\usepackage[nameinlink,noabbrev,capitalize]{cleveref}
\numberwithin{equation}{section}

\setlist[enumerate]{label=\textup{(\roman*)},leftmargin=*}
\newtheorem{theorem}{Theorem}[section]
\newaliascnt{proposition}{theorem}
\newtheorem{proposition}[proposition]{Proposition}
\aliascntresetthe{proposition}
\newaliascnt{lemma}{theorem}
\newtheorem{lemma}[lemma]{Lemma}
\aliascntresetthe{lemma}
\newaliascnt{corollary}{theorem}
\newtheorem{corollary}[corollary]{Corollary}
\aliascntresetthe{corollary}

\theoremstyle{definition}
\newaliascnt{definition}{theorem}

\aliascntresetthe{definition}
\newaliascnt{remark}{theorem}
\newtheorem{remark}[remark]{Remark}
\aliascntresetthe{remark}
\newaliascnt{example}{theorem}
\newtheorem{example}[example]{Example}
\aliascntresetthe{example}
\newaliascnt{conjecture}{theorem}

\aliascntresetthe{conjecture}
\newaliascnt{question}{theorem}
\newtheorem{question}[question]{Question}
\aliascntresetthe{question}
\newaliascnt{problem}{theorem}

\aliascntresetthe{problem}
\newaliascnt{alg}{theorem}

\aliascntresetthe{alg}

\allowdisplaybreaks[4]

\SetKwInput{KwInput}{Input}                
\SetKwInput{KwOutput}{Output}              

\newcommand{\Z}{\mathbb Z}
\newcommand{\R}{\mathbb R}

\newcommand{\conv}{\operatorname{conv}}
\newcommand{\aff}{\operatorname{aff}}
\newcommand{\lin}{\operatorname{lin}}
\newcommand{\relint}{\operatorname{relint}}
\newcommand{\codeg}{\operatorname{codeg}}
\newcommand{\Cay}{\operatorname{Cayley}}
\newcommand{\pyr}{\operatorname{pyr}}
\newcommand{\ULC}{\operatorname{ULC}}
\newcommand{\Ehr}{\operatorname{Ehr}}
\newcommand{\Vol}{\operatorname{Vol}}
\newcommand{\one}{\mathbf 1}

\title{Ehrhart Properties under Operations on Lattice Polytopes}

\author{Feihu Liu
\\[2mm]
{\small Center for Combinatorics, LPMC,}\\[-0.8ex]
{\small Nankai University, Tianjin 300071, P.R.~China}\\
{\small Email address: \url{liufeihu7476@163.com}}\\
}

\date{\today}

\begin{document}

\maketitle

\begin{abstract}
This paper investigates the preservation of three classes of Ehrhart properties under various operations on lattice polytopes. These operations include Cartesian products, lattice joins, lattice pyramids, free sums, Minkowski sums, Cayley sums, reflexive polarity, and integral dilations. Specifically, we consider the following properties frequently studied in Ehrhart theory:

(i): Positivity of Ehrhart coefficients, including Ehrhart positivity and magic positivity. In particular, we present lattice point counting formulas for the polytopes generated by these operations.

(ii): Coefficient properties of the $h^*$-polynomial, including symmetry, unimodality, log-concavity, ultra log-concavity, real-rootedness, and $\gamma$-positivity.

(iii): Geometric properties, including the spanning property, the integer decomposition property, very ampleness, and the existence of unimodular triangulation, regular unimodular triangulation, and quadratic triangulation.

We determine which properties are preserved under these eight operations, establishing preservation theorems or constructing explicit counterexamples. Furthermore, when a property is not preserved in general, we investigate sufficient or equivalent conditions for its preservation.
\end{abstract}

\noindent
\begin{small}
\end{small}

\noindent
\begin{small}
\end{small}

\tableofcontents

\section{Introduction}\label{sec:definitions}

\subsection{Basic concepts}

A \emph{lattice polytope} in $\R^d$ is the convex hull of a finite nonempty subset of $\Z^d$. 
Its dimension is its affine dimension; a lattice $d$-polytope has dimension $d$. 
Lattice points are points with integer coordinates. 

When some ploytopes are lower dimensional $P\subset\R^N$ with $d=\dim P<N$, their lattice points are counted in the displayed ambient space. When full dimensional coordinates are needed, we use an affine map from $\aff(P)$ to $\R^d$ that maps
$\aff(P)\cap\Z^N$ bijectively onto $\Z^d$. We call these \emph{integer affine coordinates} on $\aff(P)$.

For a lattice $d$-polytope $P\subset\R^d$, its \emph{lattice point enumerator}
$$ L_P(n)=|nP\cap\Z^d|,\quad (n\in \mathbb{N})$$
counts the number of integer points in the $n$-th dilation $nP = \{ n\alpha : \alpha \in P \}$ of $P$.
For $n=0$, the dilate is the singleton $\{0\}$.
Ehrhart~\cite{Ehrhart} proved that the function $L_P(n)$ is a polynomial in $n$ of degree $d$ with constant term $1$.
This polynomial is also called the \emph{Ehrhart polynomial} of $P$.
The \emph{Ehrhart series} of $P$ is $\operatorname{Ehr}_P(z)=\sum_{n\geq0}L_P(n)z^n$.
It has the rational form
\begin{equation*}
\operatorname{Ehr}_P(z)=\frac{h_P^*(z)}{(1-z)^{d+1}},\qquad h_P^*(z)=\sum_{i=0}^{d}h_i^*(P)z^i.
\end{equation*}
The numerator $h_P^*(z)$ is the \emph{$h^*$-polynomial} of $P$, and its coefficient vector is the \emph{$h^*$-vector}. 

The conversion formulas between the Ehrhart polynomial and the $h^*$-polynomial are as follows:
\begin{equation}\label{eq:conversion}
 L_P(n)=\sum_{i=0}^{d}h_i^*(P)\binom{n+d-i}{d},\qquad h_i^*(P)=\sum_{j=0}^{i}(-1)^j\binom{d+1}{j}L_P(i-j).
\end{equation}

In some places, we also use the following two notions. The \emph{degree} of $P$ is $s_P=\deg h_P^*(z)$, and its
\emph{codegree} is
\[r_P=\codeg(P)=d+1-s_P=\min\{n\geq1:\relint(nP)\cap\Z^d\ne\varnothing\},
\]
where $\relint(P)$ denote its relative interior. 
In particular, the degree of $P$ need not equal its dimension.
The Ehrhart--Macdonald reciprocity theorem \cite{BR} states that
\[|\relint(nP)\cap\Z^d|=(-1)^dL_P(-n)\qquad(n\geq1).
\]
Furthermore, the polynomial $h^*(x)$ satisfies \cite[Lemma 3.13; Corollary 3.16; Exercise 4.9]{BR}
\begin{align*}
h_0^*(P)=1,\quad h_1^*(P)=|P\cap\Z^d|-d-1,\quad h_d^*(P)=|\relint(P)\cap\Z^d|,\quad \sum_{i=0}^d h_i^*(P)=\Vol(P),
\end{align*}
where $\Vol(P)=d!\operatorname{vol}_d(P)$ is normalized volume, with $\operatorname{vol}_d$ denoting ordinary Euclidean volume in $\R^d$.
For a lower dimensional polytope, use integer affine coordinates on its affine span. The leading coefficient of $L_P(n)$ is
$\Vol(P)/d!$, not $\Vol(P)$.

For later explicit calculations, we also recall the fundamental parallelepiped formula \cite[Chapter~3]{BR}. If
$S=\conv(v_0,\ldots,v_d)$ is a lattice simplex, set
\[\Pi_S=\left\{\sum_{i=0}^d\lambda_i(v_i,1):0\leq\lambda_i<1\right\}.
\]
Then $h_S^*(z)$ is the sum of $z^k$ over the lattice points of $\Pi_S$, where $k$ is the last coordinate. There are exactly $\Vol(S)$ such points.

Let $P \subset \R^d$ be a lattice $d$-polytope containing the origin in its relative interior, $0 \in \relint(P)$. 
The \emph{polar dual} of $P$ is defined by
\[P^* = \{y \in \R^d : \langle y, x \rangle \leq 1 \text{ for all } x \in P\},
\]
where $\langle \cdot, \cdot \rangle$ denotes the standard Euclidean inner product. 
We say $P$ is \emph{reflexive} if $P^*$ is also a lattice polytope, meaning all vertices of $P^*$ lie in $\Z^d$. 
Equivalently, $P$ is reflexive if every facet of $P$ is supported by a hyperplane of the form $\langle u, x \rangle = 1$ for some primitive normal vector $u \in \Z^d$.

More generally, a lattice polytope is defined to be \emph{Gorenstein of index $r$} if its $r$-th dilation $rP$ possesses a unique relative interior lattice point $v$, and the translation $rP - v$ is reflexive with respect to the induced lattice on its linear span. Hibi's palindromic theorem \cite{Hibi,BRun} establishes that $P$ is Gorenstein if and only if its $h^*$-polynomial $h_P^*(z)$ is symmetric. In this case, the index coincides with the codegree, given by $r_P = d + 1 - s_P$. As a fundamental consequence, a polytope is reflexive if and only if it is Gorenstein of index $r_P = 1$, which forces $s_P = d$.

For further background on lattice polytopes, we refer to several excellent books~\cite{BR}, \cite[Chapter 4]{RP.Stanley}, and \cite{Ziegler}.

\subsection{Motivation}

Our main motivation comes from the following well-known open problem, which was recently resolved negatively: whether the $h^*$-polynomial of a lattice polytope with the integer decomposition property (IDP) is unimodal.
Adiprasito, Papadakis, Petrotou, and Steinmeyer~\cite{APPS22} proved that every Gorenstein IDP lattice polytope has a unimodal $h^*$-polynomial, also see~\cite{APP}.
Ferroni~\cite{Ferroni2026} constructed a family of counterexamples via the Cayley sum of polytopes, showing that the $h^*$-polynomial of a lattice polytope with the integer decomposition property is not necessarily unimodal.

Although most of the results in this paper had already been obtained before the author became aware of Cayley sum, the construction of Ferroni directly motivated the author to complete the paper.
This prompted the author to ask whether the Cayley sum operation on two lattice polytopes $P$ and $Q$ yields a lattice polytope that inherits some of the ``nice properties'' of $P$ and $Q$.
The term ``nice properties'' here refers mainly to the following three classes of properties often studied in Ehrhart theory for lattice polytopes: \emph{positivity of Ehrhart coefficients}, \emph{coefficient properties of the $h^*$-polynomial}, and \emph{geometric properties}.

Under two different choices of bases, one writes 
$$L_P(n)=c_dn^d+\cdots+ c_1n+ c_0=\sum_{i=0}^d b_i n^i(n+1)^{d-i}.$$
The \textbf{positivity} referred to here concerns the following two widely studied properties of Ehrhart polynomials.
\begin{enumerate}
\item A lattice polytope $P$ of dimension $d$ is said to be \emph{Ehrhart positive} if $c_i>0$ for all $i$.

\item A lattice polytope $P$ of dimension $d$ is said to be \emph{magic positive} if $b_i \ge 0$ for all $i$.
\end{enumerate}

It is well known \cite[Corollary 3.20; Theorem 5.6]{BR} that $c_d$ equals the volume of $P$, $c_{d-1}$ equals half of the boundary volume of $P$, and the constant term is $c_0=1$. Therefore, these three coefficients are always positive.
The remaining coefficients are more intricate and lack a straightforward geometric interpretation \cite{McMullen77}.
Reeve's tetrahedra provide the classical non-Ehrhart positive examples \cite{Reeve}.
We refer to \cite{Liu} for a good survey on Ehrhart positivity.

For the magic positivity, the magic coefficients are unique. They can be read from the polynomial
\begin{equation*}
 (1-x)^d L_P\!\left(\frac{x}{1-x}\right)=\sum_{i=0}^d b_ix^i.
\end{equation*}
For an Ehrhart polynomial, $a_0=f(0)=1$. 
The magic positivity implies Ehrhart positivity. 
Magic positivity also implies real-rootedness of $h_P^*(z)$ by the result of Br\"and\'en \cite{Branden06}. 
The converse is false; a unimodular triangle already has $h^*(z)=1$ and is not magic positive.
Some results on the magic positivity of polytopes are found in \cite{Avila26,Ferroni-Morales2026,LiuZhangMagic,Konoike,HillTrinh}.

Stanley~\cite{StanleyDec} proved that the coefficients of $h^*$-polynomial are nonnegative integers.
Regarding the \textbf{coefficient properties of the $h^*$-polynomial}, we mainly mention the following properties.
Let $h(z)=\sum_{i=0}^s a_i z^i$, where $a_i\geq0$ and $a_0a_s>0$. We use the following conventions.
\begin{enumerate}
\item $h(z)$ is \emph{symmetric} if $a_i=a_{s-i}$ for all $i$, or equivalently $z^sh(z^{-1})=h(z)$. Symmetry is about the actual degree $s$, not about an ambient dimension obtained by appending zeros.

\item $h(z)$ is \emph{unimodal} if $a_0\leq\cdots\leq a_k\geq\cdots\geq a_s$ for some $k$.

\item $h(z)$ is \emph{log-concave} if it has no internal zeros and $a_i^2\geq a_{i-1}a_{i+1}$ for $1\leq i<s$. ``No internal zeros'' means that the indices with positive coefficients form an interval.

\item For an integer $D \geq s$, $h(z)$ belongs to $\ULC(D)$ if it has no internal zeros and the sequence
   \[\left( \frac{a_i}{\binom{D}{i}} \right)_{i=0}^s\]
    is log-concave. Unless specified otherwise, when a $d$-polytope $P$ is described as \emph{ultra log-concave}, we implicitly mean $h_P^*(z) \in \ULC(d)$. We explicitly declare the parameter whenever the stronger condition $h_P^* \in \ULC(s_P)$ is under consideration.

\item $h(z)$ is \emph{real-rooted} if all its complex zeros are real. For a nonconstant $h^*$-polynomial these zeros are negative. The constant polynomial is real-rooted by convention.

\item A symmetric polynomial of degree $s$ is \emph{$\gamma$-positive} if its unique expansion 
     \[ h(z)=\sum_{j=0}^{\lfloor s/2\rfloor}\gamma_j z^j(1+z)^{s-2j}\]
      has $\gamma_j\geq0$ for every $j$. Thus symmetry is part of $\gamma$-positivity.
\end{enumerate}

By Newton's inequalities \cite[Section~5]{RP.StanleyAC}, we know the following implicit relation
\[ \text{real-rootedness}\ \Longrightarrow\ \ULC(s) \ \Longrightarrow\ \text{log-concavity}\ \Longrightarrow\ \text{unimodality}.
\]
Moreover, $\gamma$-positivity implies symmetric unimodality. A symmetric real-rooted polynomial with nonnegative coefficients is $\gamma$-positive; see \cite{StanleyLC}. 
We refer to \cite{Bra16,FH} as good surveys on the coefficient properties of the $h^*$-polynomial.

For a lattice $d$-polytope $P\subset\R^d$, set $p_0\in P\cap\Z^d$. 
A \emph{lattice triangulation} of $P$ is a finite collection $\mathcal{T}_P$ of lattice simplices, closed under taking faces, whose union is $P$ and whose pairwise intersections are common faces.
Regarding the \textbf{geometric properties} of lattice polytope, we mainly mention the following properties.
It includes both the decomposition property and the triangulation property.
\begin{enumerate}
\item The polytope $P$ is \emph{spanning} if the subgroup generated by $\{u-p_0:u\in P\cap\Z^d\}$ is $\Z^d$.

\item We say that $P$ possesses the \emph{integer decomposition property} (IDP) if, for any positive integer $n$, every lattice point $u\in nP\cap\Z^d$ can be expressed as a sum $u = u_1 + \cdots + u_n$, where each $u_i \in P\cap\Z^d$.

\item The polytope $\mathcal{P}$ is called \emph{very ample} if it satisfies the integer decomposition property for all sufficiently large $n$.

\item $\mathcal{T}_P$ is \emph{unimodular} if every maximal simplex has normalized volume one.

\item $\mathcal{T}_P$ is \emph{regular} if it is obtained by lifting the vertices of the triangulation to real heights and projecting the lower faces of their convex hull. These vertices may include nonvertex lattice points of $P$. 
    
\item $\mathcal{T}_P$ is termed \emph{flag} if its underlying simplicial complex is a flag complex, i.e., every minimal non-face has cardinality two.

\item $\mathcal{T}_P$ is called a \emph{quadratic triangulation} if it is simultaneously regular, unimodular, and flag. This is stronger than merely requiring a quadratic generating set for a toric ideal.
\end{enumerate}

By \cite[Theorem~2.1]{FH} and \cite{BG}, we know the following implicit relation
\begin{align*}
 \text{quadratic triangulation}&\Longrightarrow\text{regular unimodular triangulation} 
 \\ &\Longrightarrow\text{unimodular triangulation}
 \\ &\Longrightarrow\text{IDP}\Longrightarrow\text{very ample}\Longrightarrow\text{spanning}.
\end{align*}
It is worth pointing out that each of the above implications fails in the reverse direction, see \cite[Section~2.1]{FH}.
For a detailed treatment on triangulations, we refer to~\cite{DeLoera2010,HPPS}.

\subsection{Main results}

Throughout the remainder, we shall assume the following definition unless explicitly stated otherwise.
Let $P\subset\R^p$ and $Q\subset\R^q$ be a lattice $p$-polytope and a lattice $q$-polytope, respectively.
Write $e_i$ for a standard basis vector and $\Delta_d=\conv(0,e_1,\ldots,e_d)$ for a unimodular simplex.

Beyond the Cayley sum for two polytopes, we also systematically study the preservation properties of the following eight classes of operations on lattice polytopes:
\begin{enumerate}
\item Cartesian product: $P\times Q=\{(x,y):x\in P,\ y\in Q\}\subset\R^{p+q}$.

\item Lattice join: $P*Q=\conv\bigl((P,0,0)\cup(0,Q,1)\bigr)\subset\R^{p+q+1}$.

\item Lattice pyramid: $\pyr(P)=\conv\bigl((P,0),(0,1)\bigr)\subset\R^{p+1}$.

\item Free sum: $P\oplus Q=\conv\bigl((P,0)\cup(0,Q)\bigr)\subset\R^{p+q}$.

\item Minkowski sum: $P+Q=\{x+y:x\in P,\ y\in Q\}\subset\R^{d}$, where $p=q=d$.

\item Cayley sum: $\Cay(P,Q)=\conv\bigl((P\times\{0\})\cup(Q\times\{1\})\bigr)\subset\R^{d+1}$, where $p=q=d$.

\item Reflexive polarity: $P^*=\{y\in\R^p:\langle y,x\rangle\leq1\text{ for every }x\in P\}\subset\R^{p}$.

\item Integral dilations: $rP=\{rx:x\in P\}\subset\R^{p}$.
\end{enumerate}

Throughout this paper, whenever we state that an operation \emph{preserves} a certain property, we mean that if all constituent polytopes (input) possess the property, the resulting polytope (output) inherits it as well. For unary operations, this implication naturally applies to the single initial polytope. Furthermore, we do not require the constituent polytopes (input) to be of uniform dimension.

For Minkowski and Cayley operations, the constituent polytopes are embedded within a common ambient space $\R^d$ equipped with the standard lattice $\Z^d$, though they are not required to be full dimensional. 
In the context of free sums, we adopt the standard coordinates in $\R^{p+q}$, imposing the condition that both underlying polytopes contain the origin. Furthermore, a pyramid always has lattice height one.

The Cartesian product and lattice join of two lattice polytopes have been extensively studied previously; 
see Ferroni--Higashitani\cite{FH} and Liu--Tao--Xin \cite{LTX}. For completeness, we also mention these results in this paper.

Although a geometric operation on lattice polytopes may possess a straightforward definition, its induced action on the underlying discrete lattice points can be highly non-trivial.
This paper systematically compares eight operations on lattice polytopes with respect to three fundamental classes of properties:
magic positivity and Ehrhart positivity of Ehrhart coefficients, coefficient properties of the $h^*$-polynomial, and geometric properties such as integer decomposition and triangulation properties.
We summarize these results in the following \cref{tab:coefficients}.

\begin{table}[htbp]
\centering\small
\begin{tabular}{@{}lcccccccc@{}}
\toprule
Property & $\times$ & $*$ & $\oplus$ & $+$ & $\Cay$ & $\pyr$ & polar & dilations\\
\midrule
Ehrhart positivity & \textcolor{blue}{Y} & N & N & N & N & N & N & \textcolor{blue}{Y}\\
Magic positivity & \textcolor{blue}{Y} & N & N & N & N & N & N & \textcolor{blue}{Y}\\
$h^*$ symmetry & \textcolor{red}{C} & \textcolor{blue}{Y} & N & N & N & \textcolor{blue}{Y} & \textcolor{blue}{Y} & N \\
$h^*$ unimodality & N & N & N & N & N & \textcolor{blue}{Y} & N & \textcolor{blue}{Y}\\
$h^*$ log-concavity & \textcolor{blue}{Y} & \textcolor{blue}{Y} & N & N & N & \textcolor{blue}{Y} & N & \textcolor{blue}{Y}\\
$h^*\in\ULC(\dim P)$ & \textcolor{blue}{Y} & \textcolor{blue}{Y} & N & N & N & \textcolor{blue}{Y} & N & \textcolor{blue}{Y} \\
$h^*$ real-rootedness & \textcolor{blue}{Y} & \textcolor{blue}{Y} & N & N & N & \textcolor{blue}{Y} & N & \textcolor{blue}{Y}\\
$h^*$ $\gamma$-positivity & \textcolor{red}{C} & \textcolor{blue}{Y} & N & N & N & \textcolor{blue}{Y} & N & N \\
Spanning & \textcolor{blue}{Y} & \textcolor{blue}{Y} & \textcolor{blue}{Y} & N & N & \textcolor{blue}{Y} & N & \textcolor{blue}{Y}\\
Very ample & \textcolor{blue}{Y} & N & N & N & N & N & N & \textcolor{blue}{Y}\\
Integer decomposition property (IDP) & \textcolor{blue}{Y} & \textcolor{blue}{Y} & N & N & N & \textcolor{blue}{Y} & N & \textcolor{blue}{Y}\\
Unimodular triangulation & \textcolor{blue}{Y} & \textcolor{blue}{Y} & N & N & N & \textcolor{blue}{Y} & N & \textcolor{blue}{Y}\\
Regular unimodular triangulation & \textcolor{blue}{Y} & \textcolor{blue}{Y} & N & N & N & \textcolor{blue}{Y} & N & \textcolor{blue}{Y}\\
Quadratic triangulation & \textcolor{blue}{Y} & \textcolor{blue}{Y} & N & N & N & \textcolor{blue}{Y} & N & \textcolor{blue}{Y}\\
\bottomrule
\end{tabular}
\caption{A guide to the preservation results}\label{tab:coefficients}
\end{table}

In \cref{tab:coefficients}, the notation ``Y'' indicates that a property is preserved under the established conventions, whereas ``N'' signifies that it is not, supported by explicit counterexamples provided below. The symbol ``C'' denotes preservation contingent upon the equal-codegree hypothesis; we also supply a counterexample illustrating the failure of preservation when this condition is relaxed. 

For the free sum operation, the origin is permitted to reside on the boundary of the polytope. Stronger results derived via the free sum product formula are detailed separately. The application of polarity is restricted to reflexive polytopes.

For the properties in \cref{tab:coefficients} that are not preserved, we also study the sufficient or equivalent conditions for their preservation. We present the counting formulas for the above operations, namely, Ehrhart polynomials or $h^*$-polynomials.
In constructing some explicit counterexamples, we also used the software package LattE~\cite{DeLoera04} to compute and verify the corresponding Ehrhart polynomials and $h^*$-polynomials.

This paper is clearly organized. \cref{sec:product,sec:join-pyramid,Section-Pyr,sec:free,sec:minkowski,sec:cayley,sec:polarity,sec:dilation} are respectively devoted to the preservation properties under the eight polytope operations mentioned above. Finally, \cref{Section-Finily} contains the concluding remarks. \cref{Append-Section} contains the proof of \cref{thm:dilation-ulc-general}, which is relatively technical.

\section{Cartesian products}\label{sec:product}

Let $P\subset\R^p$ and $Q\subset\R^q$ be a lattice $p$-polytope and a lattice $q$-polytope, respectively.
The \emph{Cartesian product} of $P$ and $Q$ is
\[P\times Q=\{(x,y):x\in P,\ y\in Q\}\subset\R^{p+q}.
\]
Then $P\times Q$ is a lattice polytope of dimension $p+q$. Moreover, its Ehrhart polynomial is given by
\begin{align*}
L_{P\times Q}(n)=L_P(n)L_Q(n).
\end{align*}
For example, the reader is referred to \cite[Lemma 2.3]{Hibi-Higashitani-Tsuchiya-Yoshida} and \cite{FH}.
In particular, the normalized volume of a product is 
\[\Vol(P\times Q)=\binom{p+q}{p}\Vol(P)\Vol(Q).
\]

For power series $A(z)=\sum_{n\geq 0} a_nz^n$ and $B(z)=\sum_{n\geq 0} b_nz^n$, their \emph{Hadamard product} is $A(z)\mathbin{\circledast}B(z)=\sum_{n\geq 0} a_nb_nz^n$. Thus $\Ehr_{P\times Q}(z)=\Ehr_P(z)\mathbin{\circledast}\Ehr_Q(z)$.
For a nonzero polynomial $f$ define $W(f)$ by
\[\sum_{n\geq0}f(n)z^n=\frac{W(f)(z)}{(1-z)^{1+\deg f}}.
\]
In particular, $W(L_P)=h_P^*$. If $h_P^*(z)=\sum_{i=0}^p a_i z^i$, $h_Q^*(z)=\sum_{j=0}^q b_jz^j$, then
\[[z^k]h_{P\times Q}^*(z)=\sum_{i=0}^{p}\sum_{j=0}^{q}a_i b_j\binom{p-i+j}{k-i}\binom{q-j+i}{k-j}.
\]
This coefficient formula is recorded in \cite[Lemma~2.1]{LM}.

In what follows, we will frequently use the following lemma.

\begin{lemma}\label{lem:ulc-order}
If $s\leq D\leq D'$ and $h(z)\in\ULC(D)$, then $h(z)\in\ULC(D')$.
\end{lemma}
\begin{proof}
We may assume $s \geq 2$, as the log-concavity conditions are satisfied when $s \leq 1$. 
By the definition of $\ULC(D)$, we obtain the inequality
\[a_i^2 \geq a_{i-1}a_{i+1} \frac{\binom{D}{i}^2}{\binom{D}{i-1}\binom{D}{i+1}} 
= a_{i-1}a_{i+1} \left( \frac{i+1}{i} \right) \! \left( 1 + \frac{1}{D-i} \right)
\]
for each $1 \leq i < s$. Observe that the factor $1 + \frac{1}{D-i}$ is monotonically decreasing with respect to $D$. 
Thus, we have $h(z) \in \ULC(D')$.
\end{proof}

\begin{theorem}\label{thm:hadamard}
Let $P\subset\R^p$ and $Q\subset\R^q$ be a lattice $p$-polytope and a lattice $q$-polytope, respectively. Then the following statement holds.
\begin{enumerate}
\item If both $h_P^*(z)$ and $h_Q^*(z)$ are real-rooted, then $h_{P\times Q}^*(z)$ is real-rooted \emph{\cite[Theorem~0.2]{Wagner}}.
\item If both are log-concave with no internal zeros, then so is $h_{P\times Q}^*(z)$ \emph{\cite[Corollary~5.1]{LM}}.
\item If $h_P^*\in\ULC(p)$ and $h_Q^*\in\ULC(q)$, then $h_{P\times Q}^*\in\ULC(p+q)$ \emph{\cite[Theorem~1.2]{BFJ}}.
\item If both are symmetric and $r_P=r_Q$, then $h_{P\times Q}^*(z)$ is symmetric. If both are also $\gamma$-positive, then so is $h_{P\times Q}^*(z)$ \emph{\cite[Theorem~1.3]{BFJ}}.
\end{enumerate}
\end{theorem}

The defect in \cite{BFJ} is $\deg L_P-\deg h_P^*=r_P-1$. Hence the compatibility in part~(iv) is equality of codegrees, rather than equality of dimensions. If both dimensions are positive and both inputs are Gorenstein, their product is Gorenstein if and only if their indices agree; see \cite[Theorem~4.3]{LTX}.

Part~(iii) should not be restated with the actual numerator degree as its output order. From $\ULC(s_P)$ and $\ULC(s_Q)$ one obtains the hypotheses of part~(iii) by \cref{lem:ulc-order}, but its conclusion is $\ULC(p+q)$, which can be weaker than $\ULC(s_{P\times Q})$. 

Ehrhart positivity is preserved because $L_{P\times Q}(n)=L_P(n)L_Q(n)$. Magic positivity is preserved as well: a product of basis terms is
\[n^i(n+1)^{p-i}\,n^j(n+1)^{q-j}=n^{i+j}(n+1)^{p+q-i-j},
\]
so the new magic coefficients are sums of products of nonnegative magic coefficients. This result is also recorded in \cite[Lemma~4.3]{BFJ}.

\begin{theorem}{\em (\cite[Proposition~2.2]{FH})}\label{thm:product-geometry}
Cartesian products preserve spanning, very ampleness, IDP, unimodular triangulations, regular unimodular triangulations, and quadratic triangulations.
\end{theorem}

For \cref{thm:product-geometry}, the triangulation construction is the staircase product construction in \cite[Proposition~2.11]{HPPS}. The product construction uses a common ordering of vertices to triangulate each product of simplices. It is a statement using the usual integer points $\Z^{p+q}$ in $\R^{p+q}$.

We now consider the counterexamples to symmetry and unimodality under the Cartesian product.
\begin{example}
The intervals $[0,1]$ and $[0,2]$ have numerators $1$ and $1+z$, respectively. Both are symmetric and $\gamma$-positive, but
\[L_{[0,1]\times[0,2]}(n)=(n+1)(2n+1),\qquad h_{[0,1]\times[0,2]}^*(z)=1+3z.
\]
The product is not symmetric and therefore is not $\gamma$-positive. Their codegrees are two and one, as required by the preceding criterion.
\end{example}

Unimodality can also fail. It is essential here to use numerator
polynomials realized by lattice polytopes, rather than arbitrary unimodal
sequences. 
\begin{example}(Balletti; \cite[Proposition~4.3 and Theorem~4.4]{Balletti})
For $k\geq2$, $d=2k-1$, and positive integers $a,b$, set
\[ U_{a,b,k}=\conv(0,e_1,\ldots,e_{d-1},u,v),\]
where 
$$u=(\underbrace{1,\ldots,1}_{k-1},\underbrace{a,\ldots,a}_{k-1},a+1) \quad\text{and}\quad v=(\underbrace{0,\ldots,0}_{k-1},\underbrace{-b,\ldots,-b}_{k}).$$
Then
\[
 h_{U_{a,b,k}}^*(z)=1+b(z+\cdots+z^{k-1})+(a+b)z^k.
\]
Take $P=U_{13,1,13}$ and $Q=U_{8,1,14}$, of dimensions 25 and 27.
Both numerators are unimodal. For their 52-dimensional product $P\times Q$, the three coefficients at indices 22, 23, and 24 are
\[28264217617840835\ >\ 28180199302315925\ <\ 28214521222404290.
\]
This is a polytope counterexample to unimodality preservation.
\end{example}

\section{Lattice joins}\label{sec:join-pyramid}

The \emph{standard lattice join} of $P$ and $Q$ is
\[P*Q=\conv\bigl((P,0,0)\cup(0,Q,1)\bigr)\subset\R^{p+q+1}.
\]
Then $P* Q$ is a lattice polytope of dimension $p+q+1$. Moreover, its Ehrhart polynomial and Ehrhart series are given by
\begin{equation*}
 L_{P*Q}(n)=\sum_{k=0}^nL_P(k)L_Q(n-k),\qquad \Ehr_{P*Q}(z)=\Ehr_P(z)\Ehr_Q(z).
\end{equation*}
In particular, its $h^*$-polynomial is $h_{P*Q}^*(z)=h_P^*(z)h_Q^*(z)$.
The normalized volume of a join is 
\[\Vol(P*Q)=\Vol(P)\Vol(Q).
\]
The corresponding degree and codegree formulas are
\[s_{P*Q}=s_P+s_Q,\quad r_{P*Q}=r_P+r_Q.
\]
These are the classical join formulas; see \cite[Lemma~1.3]{HT} and \cite[Proposition~1.2]{LTX}. 
An arbitrary lattice realization of a combinatorial join need not satisfy these identities.

Some preservation properties of the $h^*$-polynomial can be obtained from the following result.
\begin{theorem}
Let $f(z)$ and $g(z)$ be two polynomials. Then the following result holds.
\begin{enumerate}
\item If both $f(z)$ and $g(z)$ are symmetric, then $f(z)g(z)$ is symmetric. 
\item If both $f(z)$ and $g(z)$ are symmetric unimodal, then $f(z)g(z)$ is symmetric unimodal \emph{\cite{StanleyLC}}.
\item If both $f(z)$ and $g(z)$ are nonnegative log-concave with no internal zeros, then $f(z)g(z)$ is log-concave \emph{\cite{StanleyLC}}.
\item If $f(z)\in\ULC(A)$ and $g(z)\in\ULC(B)$, then $f(z)g(z)\in\ULC(A+B)$ \emph{\cite{Liggett}}.
\item If both $f(z)$ and $g(z)$ are real-rooted, then $f(z)g(z)$ is real-rooted \emph{\cite{StanleyLC}}.
\item If both $f(z)$ and $g(z)$ are $\gamma$-positive, then $f(z)g(z)$ is $\gamma$-positive (by definition).
\end{enumerate}
\end{theorem}

These facts give the following statement.

\begin{corollary}\label{cor:join-shape}
Standard lattice joins preserve symmetry, log-concavity with no internal zeros, real-rootedness, and $\gamma$-positivity of $h^*$-polynomial. They preserve ultra log-concavity both in dimension order and in actual-degree order. They also
preserve the simultaneous conditions of symmetry and unimodality.
\end{corollary}

For dimension order, first obtain $\ULC(p+q)$ and then use \cref{lem:ulc-order} to pass to the join dimension $p+q+1$.
For actual-degree order, use $s_{P*Q}=s_P+s_Q$.
There is no equal-codegree restriction for joins.

The following example shows that if $h^*_P(z)$ and $g^*_Q(z)$ are merely unimodal, then $h_{P*Q}^*(z)$ need not be unimodal.

\begin{example}(Liu-Tao-Xin; \cite[Theorem~4.4]{LTX})\label{ex:join-unimodal}
Ferroni--Higashitani \cite[Theorem~5.5]{FH} describe a five-dimensional polytope $P\subset\R^7$ whose vertices are the columns of
\[\begin{pmatrix}
1&1&1&1&0&0&0&0&0&0&0&0\\
1&1&0&0&1&1&0&0&0&0&0&0\\
0&0&0&0&1&1&1&1&0&0&0&0\\
0&0&0&0&0&0&1&1&1&1&0&0\\
0&0&0&0&0&0&0&0&1&1&1&1\\
0&0&1&1&0&0&0&0&0&0&1&1\\
111&112&0&1&0&1&0&1&0&1&0&1
\end{pmatrix}.\]
The $h^*$-polynomial of $P$ is given by $h_P^*(z)=1+6z+6z^2+113z^3$, which is unimodal.
The we have 
\[h_{P*P}^*(z)=1+12z+48z^2+298z^3+1392z^4+1356z^5+12769z^6.
\]
The last three displayed nonconstant coefficients contain the trough $1392>1356<12769$. Thus ordinary multiplication of arbitrary unimodal $h^*$-polynomials need not preserve unimodality. The dimension of $P*P$ is 11, although its numerator has degree $6$.
\end{example}

We now consider the Ehrhart positivity and magic positivity of $P*Q$.

\begin{example}(Liu-Tao-Xin; \cite[Theorem~2.3, Example~2.4]{LTX})
Liu--Tao--Xin prove that, for every Ehrhart-positive $d$-polytope $P$ with $d\geq4$, the join $rP*[0,1]$ has a negative Ehrhart coefficient for every sufficiently large integer $r$. 
Let $P= [0,1]^4$ denote the $4$-dimensional unit cube. If $r=100$, a negative linear coefficient in $L_{rP*[0,1]}(n)$ is given by $-19139191/6$.
\end{example}

\begin{example}
The interval $I=[0,1]$ is magic positive. However, the join $I*I=\Delta_3=\conv(0,e_1,e_2,e_3)$ has
magic coefficients $(1,-7/6,1/3,0)$.
\end{example}

We now turn to the geometric properties of $P*Q$.

\begin{theorem}{\em (Liu-Tao-Xin; \cite[Sections~5--6]{LTX})}\label{thm:join-geometry}
Standard lattice joins preserve spanning, IDP, unimodular triangulations, regular unimodular triangulations, and quadratic triangulations. However, the join of two very ample lattice polytopes is not necessarily very ample.
\end{theorem}

In what follows, we present a counterexample to very ample property.

\begin{example}\label{Example-Veryample}
Let
\[ V=\conv\{(0,0,0),(0,0,1),(1,0,0),(1,0,1),(0,1,0),(0,1,1),(1,1,16),(1,1,17)\}.
\]
This polytope appears in \cite[SEction 9]{Balletti2021}.
This is a very ample non-IDP polytope; see \cite[Theorem~3.8]{FH}.
One can directly see the IDP failure. The point $(1,1,3)$ lies in $2V$: the segment in $2V$ from $(1,1,0)$ to $(1,1,16)$ contains it. The first endpoint is $(1,0,0)+(0,1,0)$ and the second is $(0,0,0)+(1,1,16)$, so both belong to $2V$.
However, $(1,1,3)$ cannot be expressed as a sum of two lattice points in $V$. Therefore, $V$ is not IDP.
According to \cite[Theorem 5.3]{LTX}, $V*V$ is not very ample.
\end{example}

\section{Pyramids}\label{Section-Pyr}

The \emph{lattice pyramid} of $P$ is
\[\pyr(P)=\conv\bigl((P,0),(0,1)\bigr)\subset\R^{p+1}.
\]
It is the standard join of $P$ and a point. Its dimension is $p+1$, and
\begin{equation}\label{eq:pyramid}
 L_{\pyr(P)}(n)=\sum_{k=0}^{n}L_P(k),\qquad h_{\pyr(P)}^*(z)=h_P^*(z).
\end{equation}
The corresponding degree and codegree formulas are
\[s_{\pyr(P)}=s_P,\quad r_{\pyr(P)}=r_P+1.
\]

The following low-dimensional example shows that $\pyr(P)$ preserves neither Ehrhart positivity nor magic positivity.

\begin{example}\label{ex:pyramid-negative}
Let $P=[0,7]^4$. Then the Ehrhart polynomial is given by $L_P(n)=(7n+1)^4$ 
This is Ehrhart positive and magic positive, since $7n+1=(n+1)+6n$. 
Its pyramid satisfies
\[L_{\pyr(P)}(n)=\frac{2401}{5}n^5+\frac{3087}{2}n^4+\frac{4753}{3}n^3+504n^2-\frac{481}{30}n+1.
\]
In particular the linear coefficient is $-2401/30+294/6+28/2+1=-481/30$.
Therefore, $\pyr(P)$ is neither Ehrhart positive nor magic positive.
So this is also a counterexample for the binary join.
\end{example}

The fact that the pyramid operation on lattice polytopes does not preserve Ehrhart positivity is also discussed in \cite[Section~4]{FH}.
Magic positivity already fails in dimension two. The interval $I=[0,1]$ is magic positive, but
\[L_{\pyr(I)}(n)=\binom{n+2}{2}=(n+1)^2-\frac12 n(n+1).
\]

According to \eqref{eq:pyramid}, pyramids preserve every listed numerator property, i.e., the coefficient property of $h^*$-polynomial.
Since $\pyr(P)$ is the standard join of $P$ and a point, they preserve spanning, IDP, and all three triangulation properties by \cref{thm:join-geometry}.

In particular, it follows from \cref{Example-Veryample} that $\pyr(V)$ is not very ample.

\section{Free sums}\label{sec:free}

Let $P\subset\R^p$ and $Q\subset\R^q$ be lattice polytopes of dimensions $p$ and $q$, respectively. 
Suppose that $P$ and $Q$ contain the origin. 
The \emph{free sum} (or \emph{direct sum}) of $P$ and $Q$ is defined by 
\[
 P\oplus Q=\conv\bigl((P,0)\cup(0,Q)\bigr)\subset\R^{p+q}.
\]
Then $P\oplus Q$ is a lattice polytope of dimension $p+q$ containing the origin.
Throughout this section the origin may lie on the boundary.
When an interior origin or reflexivity is needed, it is included in the statement.

\subsection{Preliminaries}

For $0\in P$, the \emph{cone} and the \emph{gauge} of $P$ are defined by
\[\operatorname{cone}(P)=\bigcup_{\lambda\geq0}\lambda P,\qquad \rho_P(x)=\min\{\lambda\geq0:x\in\lambda P\} \quad(x\in\operatorname{cone}(P)).
\]
Thus $\rho_P(x)$ is the first dilation at which the point $x$ enters
$P$. In particular, $\rho_P(0)=0$. If $x\ne0$, boundedness of $P$ implies $\rho_P(x)>0$; the closedness of $P$ ensures that the minimum is attained. Since $0\in P$, the sets $\lambda P$ increase with $\lambda$. Consequently, 
\begin{equation}\label{eq:gauge-membership}
 x\in\lambda P\quad\Longleftrightarrow\quad
 \rho_P(x)\leq\lambda
 \qquad(x\in\operatorname{cone}(P),\ \lambda\geq0).
\end{equation}
Indeed, if $x\in\mu P$ and $0\leq\mu\leq\lambda$, then $x/\lambda=(\mu/\lambda)(x/\mu)+(1-\mu/\lambda)0\in P$;
the cases $x=0$ or $\lambda=0$ are immediate.

A facet is a face of codimension one. Every facet $H$ of a full-dimensional lattice polytope has a supporting inequality
\[
 \langle u_H,x\rangle\leq b_H,
\]
where $u_H\in\Z^p$ is an outward primitive normal vector and $b_H\in\Z$. Here \emph{primitive} means that the greatest common
divisor of the coordinates of $u_H$ is one. The normal is determined uniquely after this choice of orientation and normalization.
Since $0\in P$, we have $b_H\geq0$. The facet contains the origin exactly when $b_H=0$. For the other facets, the positive integer
$b_H$ is their \emph{lattice distance} from the origin: the parallel lattice hyperplanes are the integer levels of $\langle u_H,x\rangle$. This distance is not Euclidean distance.

For a point in $\operatorname{cone}(P)$, every inequality with $b_H=0$ already holds. The remaining inequalities for $x\in\lambda P$
are $\langle u_H,x\rangle\leq\lambda b_H$. It follows that
\begin{equation}\label{eq:gauge-facets}
 \rho_P(x)= \max\left(\{0\}\cup
 \left\{\frac{\langle u_H,x\rangle}{b_H}:b_H>0\right\}\right).
\end{equation}
In particular, the gauge has rational values at lattice points.
Define
\[
 \ell_P=\operatorname{lcm}\{b_H:b_H>0\},
\]
where $\operatorname{lcm}$ denotes the least common multiple.
Then $\rho_P(x)\in\ell_P^{-1}\Z_{\geq0}$ for every
$x\in\operatorname{cone}(P)\cap\Z^p$. For the zero-dimensional
factor $\{0\}$, set $\ell_{\{0\}}=1$.
The parameter $\ell_P$ depends on the specified origin and is
different from the codegree $r_P$.

For example, if $P=[-a,b]$ with positive integers $a,b$, then
\[
 \rho_P(x)=
 \begin{cases}\frac{x}{b},&x\geq0,\\-\frac{x}{a},&x\leq0,\end{cases}
 \qquad \ell_P=\operatorname{lcm}(a,b).
\]
If $0$ lies in the interior of a positive-dimensional $P$, then
$\ell_P=1$ is equivalent to reflexivity. Indeed, the vertices of
the polar are $u_H/b_H$; see \cite{BR}. Since $u_H$ is primitive
and $b_H$ is a positive integer, this vector is integral exactly
when $b_H=1$. If $0$ lies on the boundary, then $\ell_P=1$ only means that the facets not containing $0$ have lattice
distance one. A height-one pyramid with its apex at $0$ satisfies this latter condition without having an interior origin.

\subsection{The weighted $h^*$-polynomials}

The series
\[
 \sum_{x\in\operatorname{cone}(P)\cap\Z^p}z^{\rho_P(x)}
\]
records the first radial level of each lattice point. Its coefficients are finite since points of gauge at most $a$ belong to the bounded set $aP$. Because the exponents are multiples of $1/\ell_P$, the series is an ordinary formal power series in $z^{1/\ell_P}$.

Stapledon's \emph{weighted $h^*$-polynomial} \cite[Proposition 2.6]{StapledonAM} is specified by
\[\sum_{x\in\operatorname{cone}(P)\cap\Z^p}z^{\rho_P(x)}=\frac{\widetilde h_P(z)}{(1-z)^p}.
\]
In fact, $\widetilde h_P(z)$ is a polynomial in $z^{1/\ell_P}$ with nonnegative integer coefficients, see \cite[Equation~(4) and Remark~2]{Stapledon}. Equivalently, one may count points on the union of facets of $\lambda P$ not containing the origin: each nonzero cone point lies on this union at its unique first level $\lambda=\rho_P(x)$. The origin supplies the level-zero term.

For a nonnegative rational number $a$, let $\lceil a\rceil$ denote
the least integer greater than or equal to $a$. Define the linear map
$\Psi$ by
\[
 \Psi\left(\sum_a c_a z^a\right)=\sum_a c_a z^{\lceil a\rceil}.
\]
Thus $\Psi$ rounds the exponents up and adds coefficients when
several exponents round to the same integer. It does not preserve
multiplication: $\Psi(z^{1/2}z^{1/2})=z$, whereas
$\Psi(z^{1/2})\Psi(z^{1/2})=z^2$.

\begin{theorem}\label{thm:free-formula}
Let $P\subset\R^p$ and $Q\subset\R^q$ be full-dimensional lattice polytopes containing the origin. Then
\[\widetilde h_{P\oplus Q}(z)=\widetilde h_P(z)\widetilde h_Q(z), \qquad h_P^*(z)=\Psi(\widetilde h_P(z)),
 \qquad h_{P\oplus Q}^*(z)=\Psi(\widetilde h_P(z)\widetilde h_Q(z)).
\]
Moreover, we have 
\begin{equation}\label{eq:free-product-criterion}
 h_{P\oplus Q}^*(z)=h_P^*(z)h_Q^*(z)\quad\Longleftrightarrow\quad \ell_P=1\quad \text{or}\quad \ell_Q=1.
\end{equation}
\end{theorem}

The weighted identities are given by Stapledon \cite[Lemma~4 and Equations~(6), (8)]{Stapledon}.
The equivalence \eqref{eq:free-product-criterion} is given by Beck, Jayawant, and McAllister\cite[Theorem~1.3]{BJM}, also \cite[Corollary~9]{Stapledon}.
In particular, if $P$ is reflexive and $0\in Q$, then the product formula holds.
Braun's earlier theorem assumes that one factor is reflexive and that the other contains $0$ in its interior \cite[Theorem~1]{Braun}.

Under the product criterion (Equation \eqref{eq:free-product-criterion}), we have 
\[\Ehr_{P\oplus Q}(z)=(1-z)\Ehr_P(z)\Ehr_Q(z).
\]
The factor $1-z$ accounts for the one-dimensional difference between the free sum and the join.
The weighted formula also gives $\Vol(P\oplus Q)=\Vol(P)\Vol(Q)$: rounding exponents does not change the sum of the coefficients.

The dimensions and ordinary Ehrhart polynomials of the factors do not determine the numerator of a general free sum.

\begin{example}{(\cite[Example~1]{Stapledon})}
The intervals $I=[-2,2]$ and $I'=[-1,3]$ are lattice polytopes. Moreover, their $h^*$-polynomials are have $h^*(z)=1+3z$. 
Their gauges give
\[\widetilde h_I(z)=1+2z^{1/2}+z,\qquad \widetilde h_{I'}(z)=1+z^{1/3}+z^{2/3}+z;
\]
According to \cref{thm:free-formula}, we have 
\[h_{I\oplus I}^*(z)=1+10z+5z^2,\qquad h_{I'\oplus I'}^*(z)=1+9z+6z^2.
\]
\end{example}

\subsection{The difference between join and free sum}

The join formula recalled in \cref{sec:join-pyramid} is
\begin{equation*}
 \Ehr_{P*Q}(z)=\Ehr_P(z)\Ehr_Q(z),\qquad h_{P*Q}^*(z)=h_P^*(z)h_Q^*(z).
\end{equation*}
Unlike the free sum formula, it has no condition on facet distances.
The following statement identifies precisely when its numerator consequences also apply to a given free sum.

\begin{corollary}\label{cor:free-join}
Let $P\subset\R^p$ and $Q\subset\R^q$ be full-dimensional lattice
polytopes containing $0$. Then
\[ h_{P\oplus Q}^*(z)=h_{P*Q}^*(z) \quad\Longleftrightarrow\quad \ell_P=1\ \text{or}\ \ell_Q=1.
\]
Under these equivalent conditions, for $h^*$-polynomial, free sums preserve symmetry, log-concavity with no internal zeros, real-rootedness, $\gamma$-positivity, and ultra log-concavity in either dimension order or actual-degree order. They also preserve the simultaneous conditions of symmetry and unimodality.
\end{corollary}

According to \eqref{eq:free-product-criterion}, we have 
\[ \Ehr_{P\oplus Q}(z)=(1-z)\Ehr_{P*Q}(z), \qquad L_{P*Q}(n)=\sum_{k=0}^n L_{P\oplus Q}(k),
\]
and
\[ s_{P\oplus Q}=s_P+s_Q,\qquad r_{P\oplus Q}=r_P+r_Q-1.
\]
The series identity follows from the respective denominators $(1-z)^{p+q+1}$ and $(1-z)^{p+q+2}$. 
The codegree identity follows by substituting the dimension and numerator degree in $r=\dim+1-\deg h^*$.

\begin{theorem}
The free sum of two lattice polytopes does not preserve the unimodality of the $h^*$-polynomial.
\end{theorem}
\begin{proof}
Let $A$ be the five-dimensional polytope in \cref{ex:join-unimodal}, expressed in integer affine coordinates.
Set $\widehat A=\conv(0,(A,1))\subset\R^6$. It is a lattice pyramid of height one. We have  
$$h_{\widehat A}^*(z)=1+6z+6z^2+113z^3.$$
Every lateral facet contains the apex $0$, and the only facet not containing $0$ is the base at height one. 
Its normal is the primitive last coordinate vector. Thus we have  $\ell_{\widehat A}=1$.
The twelve-dimensional free sum therefore has numerator
\[ h_{\widehat A\oplus\widehat A}^*(z) =1+12z+48z^2+298z^3+1392z^4+1356z^5+12769z^6.
\]
Therefore, $h_{\widehat A}^*(z)$ is unimodal, but $h_{\widehat A\oplus\widehat A}^*(z)$ is not unimodal since $1392>1356<12769$. 
Thus even the exact numerator product formula does not preserve unimodality without an additional hypothesis.
\end{proof}

There is a geometric reason for the different product formulas.
For $P*Q$, we define the project by $\pi(x,y,t)=(x,y)$. Then $\pi(P*Q)=P\oplus Q$.
For a real height $t\in[0,n]$, the defining convex combinations give
\[
 n(P*Q)\cap\{(x,y,t):t\text{ is fixed}\}=((n-t)P)\times(tQ)\times\{t\}.
\]
Using \eqref{eq:gauge-membership}, a point $(x,y)\in n(P\oplus Q)$
has exactly the following possible heights above it:
\[
 \rho_Q(y)\leq t\leq n-\rho_P(x).
\]
It follows both that
\begin{equation}\label{eq:free-gauge}
 \rho_{P\oplus Q}(x,y)=\rho_P(x)+\rho_Q(y)\leq n
\end{equation}
and that a lattice point $(x,y)\in n(P\oplus Q)$ has an integer
lift in $n(P*Q)$ if and only if
\begin{equation}\label{eq:integer-lift}
 \lceil\rho_P(x)\rceil+\lceil\rho_Q(y)\rceil\leq n.
\end{equation}
For the first assertion, membership in a real dilate of the free
sum is equivalent to the interval of possible heights being
nonempty. For the second, an integer height $k$ must satisfy
$\lceil\rho_Q(y)\rceil\leq k\leq n-\lceil\rho_P(x)\rceil$;
such an integer exists exactly under \eqref{eq:integer-lift}.

\begin{example}
Let $P=Q=[0,2]$. The point $(1,1)$ belongs to
$P\oplus Q=2\Delta_2$, but its only possible height in $P*Q$ is
$t=1/2$. Thus the projection on the ambient integer spaces is
surjective, whereas the projection on the lattice points of these
polytopes is not. In this case
\[
 h_P^*(z)=h_Q^*(z)=1+z,\qquad
 h_{P*Q}^*(z)=1+2z+z^2,\qquad
 h_{P\oplus Q}^*(z)=1+3z.
\]
Indeed, $L_{2\Delta_2}(n)=\binom{2n+2}{2}=2n^2+3n+1$, which gives the last numerator by \eqref{eq:conversion}.
Note that $h_P^*(z)$ and $h_Q^*(z)$ are symmetric and $\gamma$-positive, but $h_{P\oplus Q}^*(z)$ is not symmetric.
\end{example}

\subsection{Counterexamples}

We already know that the $h^*$-polynomial of free sum do not necessarily preserve symmetry, unimodality, or $\gamma$-positivity. We now give counterexamples to other coefficient properties.

\begin{theorem}\label{ex:free-lc}
The free sum of two lattice polytopes does not preserve the log-concavity of the $h^*$-polynomial.
\end{theorem}
\begin{proof}
Let
\[ P=Q=\conv((0,0),(2,1),(1,2)).
\]
The determinant of the two nonzero vertices is $3$, so each triangle has area $3/2$. 
We know that 
$h_P^*(z)=h_Q^*(z)=1+z+z^2$, which is log-concave with no internal zeros.
Their free sum is the four-dimensional simplex
\[ S=\conv\bigl(0,(2,1,0,0),(1,2,0,0),(0,0,2,1),(0,0,1,2)\bigr).
\]
We compute its $h^*$-polynomial by the fundamental parallelepiped formula from \cref{sec:definitions}. 
Consequently, we have 
\[ h_S^*(z)=1+2z+5z^2+z^3.
\]
The polynomial $h_S^*(z)$ is not log-concave since $2^2<1\cdot5$.
\end{proof}

In the proof of \cref{ex:free-lc}, each polytope is a translate of a reflexive triangle, but its chosen origin is a vertex. Indeed, translating by $-(1,1)$ gives the facet inequalities $-2x+y\leq1$, $x-2y\leq1$, and $x+y\leq1$.
The free sum of these translated factors has numerator $(1+z+z^2)^2$ by \cref{cor:free-join}.
This contrast is caused by the position of the origin, although the Ehrhart polynomials of $P$ and $Q$ are unchanged.

\begin{theorem}\label{ex:free-ulc}
The free sum of two lattice polytopes does not preserve the ultra log-concavity or real-rootedness of the $h^*$-polynomial.
\end{theorem}
\begin{proof}
Let
\[T=\conv((-3,-2),(-2,-3),(2,2)),\qquad I=[-2,2].
\]
The triangle is described by
\[ -x-y\leq5,\qquad 5x-4y\leq2,\qquad -4x+5y\leq2.
\]
Their respective $h^*$-polynomials are as follows
\[
 h_T^*(z)=1+4z+4z^2=(1+2z)^2,\qquad h_I^*(z)=1+3z.
\]
Both polynomials are real-rooted and ultra log-concave in their actual-degree orders, which here equal their dimension orders.

To count the rational slices needed for the free sum, fix an integer $k\geq0$ and set $s=x+y$. The inequalities for
$(k/2)T$ imply
\[ s\geq-\frac{5k}{2},\qquad \frac{5s-k}{9}\leq x\leq\frac{4s+k}{9}.
\]
The interval for $x$ is nonempty only if $s\leq2k$. Conversely,
these conditions imply all three facet inequalities.
As $x,y$ are integers exactly when $s,x$ are integers, we obtain
\[ |(k/2)T\cap\Z^2|= \sum_{s=\lceil-5k/2\rceil}^{2k} \max\left\{0,\ \left\lfloor\frac{4s+k}{9}\right\rfloor -\left\lceil\frac{5s-k}{9}\right\rceil+1\right\}.
\]
Here $\lfloor a\rfloor$ is the greatest integer at most $a$; the maximum assigns zero to an interval containing no integers.
By evaluating the finite sums, we obtain
\[ \begin{array}{c|rrrrrrr}
 k&0&1&2&3&4&5&6\\ \hline
 |(k/2)T\cap\Z^2|&1&3&7&13&22&32&46
 \end{array}.
\]
Since $\rho_I(j)=|j|/2$, Equation~\eqref{eq:free-gauge} shows that the slice of $n(T\oplus I)$ at integer last coordinate $j$ is $(n-|j|/2)T$, for $|j|\leq2n$.
The height-zero slice occurs once and every nonzero height has its negative counterpart. Hence
\[L_{T\oplus I}(n)=\sum_{j=-2n}^{2n}|(n-|j|/2)T\cap\Z^2| =|nT\cap\Z^2|  +2\sum_{k=0}^{2n-1}|(k/2)T\cap\Z^2|.
\]
At $n=0,1,2,3$, this gives $1,15,70,202$.
The dimension of $T\oplus I$ is three. Equation~\eqref{eq:conversion} therefore gives all its $h^*$ coefficients:
\begin{align*}
 h_0^*=1,\quad
 h_1^*=15-4=11,\quad
 h_2^*=70-4\cdot15+6=16,\quad
 h_3^*=202-4\cdot70+6\cdot15-4=8.
\end{align*}
Therefore, the $h^*$-polynomial $h_{T\oplus I}^*(z)$ is not ultra log-concave because $(16/3)^2\leq (11/3)\cdot8$.
Furthermore, the $h^*$-polynomial $h_{T\oplus I}^*(z)$ is not real-rooted.
\end{proof}

\begin{theorem}\label{ex:bipyramid}
The free sum of two lattice polytopes does not preserve Ehrhart positivity or magic positivity.
\end{theorem}
\begin{proof}
For $d\geq1$, define
\[ B_d=[-1,1]^d\oplus[-1,1].
\]
The cube and the interval are reflexive. They are also magic positive, since their Ehrhart polynomials are $(2n+1)^d$
and $2n+1$, and $2n+1=n+(n+1)$.
The gauges of the factors are $\max_i|x_i|$ and $|j|$.
Thus the integer slice of $nB_d$ at last coordinate $j$ is the cube $\big[-(n-|j|),n-|j|\big]^d$ whenever $|j|\leq n$.
It follows that
\[ L_{B_d}(n)=(2n+1)^d+2\sum_{k=0}^{n-1}(2k+1)^d.
\]
For $d=10$, expansion gives
\begin{align}
L_{B_{10}}(n)={}&\frac{2048}{11}n^{11}+1024n^{10} +\frac{12800}{3}n^9+11520n^8+17152n^7\notag\\
&+13440n^6+6080n^5+3360n^4+1976n^3+180n^2-\frac{4450}{33}n+1.\label{eq:bipyramid-polynomial}
\end{align}
Therefore, $L_{B_{10}}(n)$ is neither Ehrhart positive nor magic positive.
\end{proof}

A smaller example distinguishes magic positivity from ordinary Ehrhart positivity. Write
$C_d=\conv(\pm e_1,\ldots,\pm e_d)$ for the standard cross-polytope. Its standard counting formula is
\cite[Section~2.2.1]{Liu}
\[ L_{C_d}(n)=\sum_{j=0}^d2^j\binom dj\binom nj.
\]
Specializing to $d=2$ gives
\[ L_{C_2}(n)=(n+1)^2+n^2.
\]
This is magic positive, as is $L_{[-1,1]}(n)=n+(n+1)$. Their free sum is $C_3$, and the same count yields
\[ L_{C_3}(n)=\frac43n^3+2n^2+\frac83n+1 =(n+1)^3-\frac13n(n+1)^2-\frac13n^2(n+1)+n^3.
\]
Uniqueness of the magic basis expansion proves that $C_3$ is not magic positive, although its ordinary Ehrhart
coefficients are positive.

\subsection{Spanning, decomposition, and triangulations}

\begin{theorem}
If $P$ and $Q$ are spanning and contain $0$, then $P\oplus Q$ is spanning.
\end{theorem}
\begin{proof}
Selecting the origin as the base point, observe that the vectors $(x,0)$ for $x \in P \cap \Z^p$ generate the sublattice $\Z^p \times \{0\}$, while the vectors $(0,y)$ for $y \in Q \cap \Z^q$ generate $\{0\} \times \Z^q$. Because these elements are lattice points of the free sum $P \oplus Q$, their integer linear combinations generate the full ambient lattice $\Z^{p+q}$.
\end{proof}

For IDP, there is a useful sufficient condition as follow.

\begin{proposition}\label{prop:free-idp}
Suppose that $P$ and $Q$ are IDP and contain their origins.
If, for either $P$ or $Q$, every facet of not containing the origin has lattice distance one, then $P\oplus Q$ is IDP.
In particular, this holds if one factor is reflexive.
\end{proposition}
\begin{proof}
By symmetry, we may assume without loss of generality that $\ell_P=1$.
Let $(x,y)\in n(P\oplus Q)\cap\Z^{p+q}$.
By~\eqref{eq:free-gauge}, we have $\rho_P(x)+\rho_Q(y)\leq n$.
According to \eqref{eq:gauge-facets}, $\rho_P(x)$ is given by the maximum of $0$ and the integral facet functionals evaluated at $x$, as all positive facet constants are equal to one.
Thus $a=\rho_P(x)$ is an integer with $0\leq a\leq n$.
Equation~\eqref{eq:gauge-membership} now gives $x\in aP$ and $y\in(n-a)Q$.

Since $P$ and $Q$ possess the IDP, there exist lattice points $p_i \in P$ and $q_j \in Q$ such that $x = p_1 + \cdots + p_a$ and $y = q_1 + \cdots + q_{n-a}$.
When a number of summands is zero, the corresponding vector is zero and the sum is empty. 
The $n$ points $(p_1,0),\ldots,(p_a,0),(0,q_1),\ldots,(0,q_{n-a})$ belong to $P\oplus Q$ and sum to $(x,y)$.

By definition, the origin is an interior point of any reflexive polytope. This completes the proof.
\end{proof}

Hibi--Higashitani give an equivalence under additional hypotheses \cite[Theorem~1.1]{HH}. In the present coordinates their
statement assumes that both factors contain $0$ and are spanning, and that
\[ (P\oplus Q)\cap\Z^{p+q}=\bigl((P\cap\Z^p)\times\{0\}\bigr) \cup\bigl(\{0\}\times(Q\cap\Z^q)\bigr).
\]
Under these hypotheses the free sum is IDP if and only if both factors are IDP and at least one of $\ell_P,\ell_Q$ equals one.
The displayed condition excludes additional mixed lattice points of the free sum and is essential to that equivalence.
The sufficient statement in \cref{prop:free-idp} does not require this condition.

\begin{theorem}\label{ex:free-holes}
The free sum of two lattice polytopes does not preserve the very ampleness, IDP, unimodular triangulations, regular unimodular triangulations, or quadratic triangulations.
\end{theorem}
\begin{proof}
Let
\[ P=\conv(\pm2e_1,\pm3e_2),\qquad Q=[-5,5].
\]
Its boundary has exactly four lattice points. The polygon theorem of Bruns--Gubeladze--Trung \cite[Corollary~3.2.5]{BGT}, also
\cite[Proposition~2.25]{HPPS}, states that a lattice polygon with at least four boundary lattice points has a quadratic triangulation. 
The interval $Q$ has the unit-segment triangulation. It is unimodular, is regular by lifting the
integer $j$ to height $j^2$, and is flag because a path has no clique of more than two vertices. 
Thus both polytopes have all six geometric properties in \cref{tab:coefficients}.

Their gauges are $|x|/2+|y|/3$ and $|z|/5$. Equation~\eqref{eq:free-gauge} therefore gives
\[ R=P\oplus Q=\{(x,y,z):15|x|+10|y|+6|z|\leq30\}.
\]
For every $n\geq2$, the integer point $u_n=(1,2,5n-6)$ lies in $nR$, since its weighted absolute-coordinate sum is $30n-1$.

Define a function $\lambda(x,y,z)=15x+10y+6z$. We first show that no lattice point of $R$ has
$\lambda(x,y,z)=29$. 
If a coordinate were negative, its contribution to the absolute-coordinate sum exceeds its contribution to $\lambda$ by at least $12$.
The absolute-coordinate sum would then be at least $29+12>30$. Thus all coordinates must be nonnegative.
The equation $15x+10y+6z=29$ implies $z\equiv4\pmod5$ and $0\leq z\leq4$, so $z=4$.
It would follow that $15x+10y=5$, which is impossible for nonnegative integers $x,y$.

For every $a\in R\cap\Z^3$, the integer $30-\lambda(a)$ is consequently either zero or at least two. 
If $u_n=(1,2,5n-6)=a_1+\cdots+a_n$ with $a_i\in R\cap\Z^3$, by the linearity, we have
\[ 1=30n-\lambda(u_n)=\sum_{i=1}^n(30-\lambda(a_i)),
\]
which is impossible. Thus each dilation $nR$, $n\geq2$, contains a lattice point that is not a sum of $n$
lattice points of $R$. In particular, $R$ is not very ample and is not IDP. Since a unimodular triangulation
implies IDP, it has none of the three triangulation properties considered here.

The origin is interior, and $R$ contains the three standard basis vectors, so it is spanning.
\end{proof}

For comparison, the integer height of a join prevents this type of mixed allocation. The following consequence makes the distinction between IDP and very ampleness explicit.

\begin{theorem}\label{prop:join-eventual-idp}
For nonempty lattice polytopes $P,Q$, the following conditions are equivalent: (1). $P*Q$ is very ample;
(2). $P*Q$ is IDP; (3). both $P$ and $Q$ are IDP.
\end{theorem}
\begin{proof}
If both $P$ and $Q$ are IDP, then the join is IDP by \cite[Theorem~5.2]{LTX}.
We know that an IDP polytope is very ample.

For the remaining implication: $(1)\Rightarrow(3)$, suppose that $P$ is not IDP. Fix $c\geq2$ and $u\in cP\cap\Z^p$ that is
not a sum of $c$ lattice points of $P$, and fix a lattice point $q_0\in Q$. 
For every $n\geq c$, the slice formula gives
\[(u,(n-c)q_0,n-c)\in n(P*Q)\cap\Z^{p+q+1}.
\]
Each lattice point of $P*Q$ has height zero or one. A decomposition of the displayed point into $n$
such points would therefore contain exactly $n-c$ points from the $Q$ layer and $c$ points from the
$P$ layer. In its first coordinate block it would express $u$ as a sum of $c$ lattice points of $P$, a contradiction. 
It follows that such indecomposable lattice points exist in $n(P*Q)$ for all $n \ge c$, thereby proving that the join $P*Q$ is not very ample. By symmetry, an analogous argument holds if $Q$ is not IDP.
\end{proof}

Unlike free sums, where the allocation of a point between factors is dictated by the real radial levels in \eqref{eq:free-gauge}, \cref{prop:join-eventual-idp} requires no assumptions regarding the origins.

\section{Minkowski sums}\label{sec:minkowski}

Let $P,Q\subset\R^d$ be nonempty lattice polytopes. The \emph{Minkowski sum} of $P$ and $Q$ is defined by 
\[
 P+Q=\{x+y:x\in P,\ y\in Q\}.
\]
The ambient lattice is $\Z^d$ for $P$ and $Q$, and for the sum. For a lower-dimensional polytope, spanning and normalized volume are understood in the integer affine coordinates specified in \cref{sec:definitions}. 
We denote the linear span by 
\[ U=\lin(P-P)=\{x - y : x \in P, y \in P\},\qquad V=\lin(Q-Q)=\{x - y : x \in Q, y \in Q\}.
\]
If $P=\conv(p_1,\ldots,p_s)$ and $Q=\conv(q_1,\ldots,q_t)$, then
$P+Q=\conv\{p_i+q_j:1\leq i\leq s,\ 1\leq j\leq t\}$.
Indeed, the sum of convex combinations with coefficients $\lambda_i$
and $\mu_j$ is the convex combination of the $p_i+q_j$ with coefficients
$\lambda_i\mu_j$. The reverse inclusion follows from convexity.
Consequently, $P+Q$ is a lattice polytope. Moreover,
\[
 (P+Q)-(P+Q)=(P-P)+(Q-Q),\qquad
 \dim(P+Q)=\dim(U+V).
\]
In particular, if $P$ and $Q$ are full-dimensional in $\R^d$, the $P+Q$ has dimension $d$.

\subsection{Ehrhart positivity and magic positivity}

McMullen's multivariate Ehrhart theorem states that, for fixed lattice polytopes $P_1,\ldots,P_m\subset\R^d$, the function
\[ (n_1,\ldots,n_m)\longmapsto \bigl|(n_1P_1+\cdots+n_mP_m)\cap\Z^d\bigr|
\]
is a polynomial in $n_1,\ldots,n_m$ on $\Z_{\geq0}^m$; see \cite{McMullen,BR}.
Its diagonal specialization is $L_{P_1+\cdots+P_m}(n)$.
Once this polynomial is known, applying \eqref{eq:conversion} with dimension $\dim(P_1+\cdots+P_m)$ yields the $h^*$-polynomial. Furthermore, the individual Ehrhart polynomials of the summands generally fail to determine these mixed counts.

For example, for every integer $a\geq1$, the two segments $[(0,0),(1,0)]$ and $[(0,0),(1,a)]$ are primitive and both have Ehrhart
polynomial $n+1$. Their Minkowski sum is a lattice parallelogram of area $a$ with four primitive edges. Pick's theorem \cite[Theorem 2.8]{BR} therefore gives
\[ L(n)=an^2+2n+1, \qquad h^*(z)=1+az+(a-1)z^2.
\]
The second formula follows by substituting the first into \eqref{eq:conversion}. As $a$ varies, the Ehrhart polynomials $n+1$
remain fixed and the polynomial $L(n)$ changes.

Tsuchiya constructed Minkowski sums of Ehrhart positive polytopes whose Ehrhart polynomials have negative coefficients; two constructions are recorded in \cite[Examples~4.6.1--4.6.2]{Liu}. The next example modifies the first construction so that the Ehrhart polynomials of $P$ and $Q$ are magic positive.

\begin{theorem}\label{ex:minkowski-magic}
The Minkowski sum of two lattice polytopes does not preserve Ehrhart positivity or magic positivity.
\end{theorem}
\begin{proof}
Let $e_1,\ldots,e_4$ be the standard basis of $\Z^4$. For an integer $m\geq2$, set
\[ P=3\conv(0,e_1,e_2,e_3),\qquad Q_m=[0,(1,m-1,m-1,m)].
\]
The dimensions of $P$, $Q_m$, and $P+Q_m$ are three, one, and four, respectively. 
The last assertion follows because the direction of $Q_m$ has nonzero fourth coordinate. 
The segment is primitive, since the first coordinate of its direction is one. The Ehrhart polynomials are 
\[ L_P(n)=\binom{3n+3}{3} =(n+1)^3+\frac{5}{2}n(n+1)^2+n^2(n+1),\qquad 
 L_{Q_m}(n)=n+1.
\]
These are magic expansions with nonnegative coefficients.

We compute $L_{P+Q_m}(n)$ for $n\geq1$ by fixing the fourth coordinate $z\in\{0,1,\ldots,mn\}$. The corresponding slice, in its first three coordinates, is
\[ \left(\frac zm,\frac{(m-1)z}{m},\frac{(m-1)z}{m}\right) +3n\Delta_3,
 \qquad \Delta_3=\conv(0,e_1,e_2,e_3)\subset\R^3.
\]
For each of the $n+1$ heights where $z$ is a multiple of $m$, the translation vector is integral. 
Consequently, the slice contains exactly $\binom{3n+3}{3}$ lattice points.

Every remaining height is uniquely $z=m\ell+j$, with $0\leq\ell\leq n-1$ and $1\leq j\leq m-1$. The fractional parts of
the translation vector $\left(\frac zm,\frac{(m-1)z}{m},\frac{(m-1)z}{m}\right)$ are
$f_1=\frac jm$, $f_2=1-\frac jm$, and $f_3=1-\frac{j}{m}$, respectively.
Now we need to count the number of integer lattice points \((y_1,y_2,y_3)\) satisfying the following conditions:
\begin{align*}
y_1 - f_1 \ge 0,\quad
y_2 - f_2 \ge 0,\quad
y_3 - f_3 \ge 0,\quad
(y_1 - f_1) + (y_2 - f_2) + (y_3 - f_3) \le 3n.
\end{align*}
Consider the condition \(y_1 \ge f_1\). Since \(y_1\) must be an integer and \(0<f_1 = \frac{j}{m}<1\), the smallest possible value of \(y_1\) is in fact 1.
Similarly, \(y_2 \ge 1 - \frac{j}{m}\) implies that the minimum value of \(y_2\) is \(1\); \(y_3 \ge 1 - \frac{j}{m}\) implies that the minimum value of \(y_3\) is \(1\).

Set $k_i=y_i-1\geq 0$. Therefore, we have 
$$(k_1 + 1 - f_1) + (k_2 + 1 - f_2) + (k_3 + 1 - f_3) \le 3n.$$
Furthermore, we obtain 
$$k_1 + k_2 + k_3 \le 3n - \left(1 + \frac{j}{m}\right).$$
We have reduced the problem to finding three nonnegative integers \(k_1,k_2,k_3\) whose sum is at most $\lfloor 3n - 1 - \frac{j}{m} \rfloor = 3n - 2$.
The number of such triples is $\binom{3n-2+3}{3}$. There are $n(m-1)$ heights of this type. 

Thus we have 
\[\begin{aligned}
 L_{P+Q_m}(n)&=(n+1)\binom{3n+3}{3}+n(m-1)\binom{3n+1}{3}\\
 &=\frac{9m}{2}n^4+\frac{27}{2}n^3+\left(15-\frac m2\right)n^2+\frac{13}{2}n+1.
 \end{aligned}
\]
The formula also gives $L_{P+Q_m}(0)=1$. For $m=31$, the coefficient of $n^2$ is $-1/2$. 
Consequently, Minkowski sums preserve neither Ehrhart positivity nor magic positivity, even when polytopes $P$ and $Q_m$ are magic positive.
\end{proof}

Failure of Ehrhart positivity also occurs when both polytopes are full-dimensional in the same space. 

\begin{example}(Tsuchiya; \cite[Example~4.6.2]{Liu})
In $\Z^5$, set
\[ \begin{aligned}
 P_0&=\conv\bigl(0,e_1,(0,1,5,15,16),e_3,e_4,e_5\bigr),\\
 Q_0&=\conv\bigl(0,e_1,(1,1,15,15,16),e_3,e_2,e_5\bigr).
 \end{aligned}
\]
Both are five-dimensional simplices, and $P_0$ is unimodular.
Their Ehrhart polynomials are given by 
\begin{align*}
L_{P_0}(n)&=\binom{n+5}{5},\quad L_{Q_0}(n)=\frac18n^5+\frac5{12}n^4+\frac{17}{24}n^3+\frac{19}{12}n^2+\frac{13}{6}n+1,\\
L_{P_0+Q_0}(n)&=\frac{3007}{40}n^5+\frac{359}{24}n^4-\frac{255}{24}n^3+\frac{193}{24}n^2+\frac{89}{20}n+1.
\end{align*}
The $P_0$ and $Q_0$ are Ehrhart positive, whereas the cubic coefficient of $P_0+Q_0$ is negative.
\end{example}

\subsection{Coefficient properties and geometric properties}

\begin{theorem}
The Minkowski sum of two lattice polytopes does not preserve the symmetry of the $h^*$-polynomial.
\end{theorem}
\begin{proof}
The interval $[0,2]$ has Ehrhart polynomial $2n+1$ and $h^*$-polynomial $1+z$, whereas the Minkowski sum $[0,2]+[0,2]=[0,4]$ has Ehrhart polynomial $4n+1$ and $h^*$-polynomial $1+3z$. The $h^*$-polynomial is not symmetric.
\end{proof}

\begin{theorem}\label{ex:minkowski-shape}
The Minkowski sum of two lattice polytopes does not preserve the unimodality, log-concavity, ultra log-concavity, real-rootedness, or $\gamma$-positivity of the $h^*$-polynomial; nor does it preserve spanning, very ampleness, IDP, unimodular triangulations, regular unimodular triangulations, or quadratic triangulations.
\end{theorem}
\begin{proof}
In $\R^8$ with lattice $\Z^8$, set
\[ P=\conv(0,e_1,e_2,e_3,e_4),\qquad
 Q=\conv(0,e_5,e_6,e_7,2e_8-e_1-\cdots-e_7).
\]
We will show that both polytopes are unimodular four-simplices. 
However, the Minkowski sum is an eight-dimensional nonspanning polytope with a nonunimodal $h^*$-polynomial.

We consider the invertible linear map:
\[ T(x_1,\ldots,x_8) =\left(x_1+\frac{x_8}{2},\ldots,x_7+\frac{x_8}{2},\frac{x_8}{2}\right).
\]
Its image of $\Z^8$ is the lattice
\[ T(\Z^8)=\Z^8\ \cup\ \left(\Z+\frac12\right)^8.
\]
Indeed, the parity of $x_8$ determines the two cases. Conversely, a vector $y$ in either displayed set has the integer preimage
$x_i=y_i-y_8$ for $1\leq i\leq7$ and $x_8=2y_8$.
Thus $T$ must be used together with this image lattice when counting points.

The images $T(P)$ and $T(Q)$ are the standard four-simplices in the first and last four coordinate positions, respectively. 
The half-integer coset has no point on either coordinate subspace: all four coordinates in the other block would have to vanish. Consequently, the image lattice on each subspace is exactly its usual copy of $\Z^4$. This proves that both polytopes are unimodular in their integer affine coordinates. 
Each has $h^*(z) = 1$, and its trivial triangulation is quadratic: the simplex is inherently unimodular, regularity is induced by a constant lifting function, and the flag property holds since every subset of its vertices forms a face.

As a real polytope, we have $T(P+Q)=\Delta_4\times\Delta_4$, where $\Delta_4=\conv(0,e_1,e_2,e_3,e_4)\subset\R^4$. 
The two independent four-dimensional direction spaces also show that $\dim(P+Q)=8$.
In the integer coset, the $n$-th dilate contains $\binom{n+4}{4}^2$ lattice points. 
In the half-integer coset, write $y_i=k_i+1/2$ for $1\leq i\leq8$. 
The nonnegativity condition holds if and only if $k_i \geq 0$, and each four-coordinate sum is bounded above by $n$ if and only if the corresponding sum of the $k_i$ does not exceed $n-2$.
It follows that
\[L_{P+Q}(n)=\binom{n+4}{4}^2+\binom{n+2}{4}^2.
\]
The second binomial term is zero for $n=0,1$.

Therefore, we have 
\[h_{P+Q}^*(z)=1+16z+37z^2+32z^3+37z^4+16z^5+z^6.
\]
The inequalities $37>32<37$ show that it is not unimodal; in particular, $32^2<37^2$ also violates log-concavity. 
Furthermore, $h_{P+Q}^*(z)$ is neither ultra log-concave nor real-rooted.
The polynomial is symmetric, but it cannot be $\gamma$-positive because $\gamma$-positivity implies unimodality. Thus all five properties fail to be preserved, although $h_P^*(z)=h_Q^*(z)=1$.

At $n=1$ the half-integer coset contributes no point. Therefore every lattice point of $P+Q$ has even eighth coordinate in the original $x$-coordinates. On the other hand, the sum $P+Q$ contains
\[0,e_1,\ldots,e_7,2e_8-e_1-\cdots-e_7.
\]
These points generate $\Z^7\times2\Z$. The lattice-point difference
group of $P+Q$ is therefore exactly $\Z^7\times2\Z$, which has index
two in $\Z^8$. The sum is not spanning. Under the conventions of
\cref{sec:definitions}, it is consequently not very ample or IDP and
admits no unimodular, regular unimodular, or quadratic triangulation.
This completes the proof.
\end{proof}

\subsection{A sufficient condition for spanning}

In particular, if one summand is full-dimensional and spanning, the Minkowski sum is spanning for any lattice polytope, as it necessarily contains an integral translate of the full-dimensional summand. However, for lower-dimensional spanning summands, the equality in \eqref{eq:saturated-sum} may fail. \cref{Memm-Spanning} gives a sufficient condition for a Minkowski sum; it does not assert necessity.

\begin{proposition}\label{Memm-Spanning}
Suppose that $P$ and $Q$ are spanning in their respective integer affine coordinates. If
\begin{equation}\label{eq:saturated-sum}
(\Z^d\cap U)+(\Z^d\cap V)=\Z^d\cap(U+V),
\end{equation}
then $P+Q$ is spanning.
\end{proposition}
\begin{proof}
Choose $p_0\in P\cap\Z^d$ and $q_0\in Q\cap\Z^d$. Both $P+q_0$ and $p_0+Q$ are contained in $P+Q$. 
By the spanning assumptions, differences of the lattice points in the $P+q_0$ generate $\Z^d\cap U$, and
differences in the $p_0+Q$ generate $\Z^d\cap V$. Hence differences of lattice points of $P+Q$ generate a subgroup containing the left-hand side of \eqref{eq:saturated-sum}. Every such difference belongs to $\Z^d\cap(U+V)$.
The assumed equality identifies these two groups and proves that $P+Q$ is spanning in its integer affine coordinates.
\end{proof}

The loss of spanning in \cref{ex:minkowski-shape} uses lower-dimensional two lattice polytopes. 
Tow full-dimensional spanning polytopes always give a spanning Minkowski sum, but their Minkowski sum can still fail very
ampleness, even when the input polytopes are unimodular simplices.

\begin{theorem}\label{ex:full-dimensional-holes}
There exist two full-dimensional unimodular simplices such that their Minkowski sum preserves the spanning property but does not preserve very ampleness, IDP, unimodular triangulations, regular unimodular triangulations, or quadratic triangulations.
\end{theorem}
\begin{proof}
In $\Z^4$, set $v=(1,1,3,0)$ and
\[
 A=\conv(0,e_1,e_2,e_3+e_4,e_4),\qquad
 D=\conv(0,v,e_4,e_1+e_4,e_3+e_4).
\]
For $A$, by applying an elementary column operation, we obtain the standard basis. Specifically, subtracting the column $e_4$ from $e_3+e_4$ in its generator matrix. Similarly for $D$, subtracting the column $e_4$ from the final two columns produces the vectors $v, e_4, e_1$, and $e_3$.
Since $v$ is the only vector with a non-zero second coordinate, which is precisely $1$, expanding the determinant along the second row yields an absolute value of $1$. Consequently, both $A$ and $D$ are unimodular $4$-simplices.
In particular, both are spanning, IDP, and very ample, and their trivial triangulations are quadratic.

Every point of $A+D$ has nonnegative fourth coordinate. 
Its face where that coordinate vanishes is the sum of the corresponding faces of $A$ and $D$. 
In the first three coordinates this is
\[ F=\conv(0,e_1,e_2)+[(0,0,0),(1,1,3)].
\]
A point of $F$ has the form $(s,t,0)+\lambda(1,1,3)$ with $s,t\geq0$, $s+t\leq1$, and $0\leq\lambda\leq1$.
Assume $(x,y,z)\in F$.
By solving $z=3\lambda$, $s=x-\frac{z}{3}$, and $t=y-\frac{z}{3}$, we obtain the equivalent inequalities
\[ 0\leq z\leq3,\qquad 3x\geq z,\qquad 3y\geq z, \qquad 3x+3y\leq3+2z.
\]
At $z=1$, the two lower bounds force integer coordinates $x,y$ to be at least one, whereas the last inequality requires $3x+3y\leq5$.
Thus no lattice point of $F$ has third coordinate one. 
Every lattice point of $F$ has third coordinate zero, two, or three.

For every integer $n\geq2$, the point $u=(1,1,1,0)$ belongs to $n(A+D)$. To verify this, write
\[u=\left(\frac23e_1+\frac23e_2\right)+\frac13v.
\]
The first summand belongs to $2A$, since its coefficients are nonnegative and sum to $\frac{4}{3}\leq2$; the second belongs to $D$.
Because $0\in A\cap D$, one has $2A\subset nA$ and $D\subset nD$. Therefore, the claimed follows.

Suppose that $u$ were a sum of $n$ lattice points of $A+D$. 
Since their fourth coordinates are nonnegative and must sum to zero, each summand necessarily lies on the face $F$.
Their third coordinates would then be nonnegative integers, each either zero or at least two. 
They cannot sum to the third coordinate one of $u$. This proves that decomposition fails in every degree $n\geq2$. 
Consequently, $A+D$ is neither very ample nor IDP, and it admits none of the three aforementioned unimodular triangulation properties. Nevertheless, it remains spanning: the condition $0 \in D$ ensures $A \subseteq A+D$, and the differences of the lattice points in $A$ inherently generate the ambient lattice $\Z^4$.
This completes the proof.
\end{proof}

\section{Cayley sums}\label{sec:cayley}

For nonempty lattice polytopes $P,Q\subset\R^d$, the \emph{Cayley sum} (or \emph{Cayley polytop}e) of $P$ and $Q$ is defined as
\[\Cay(P,Q)=\conv\bigl((P\times\{0\})\cup(Q\times\{1\})\bigr)\subset\R^{d+1}.
\]
The lattice is $\Z^d\times\Z$. 
The two layers use the same horizontal space. We denote the linear span by $U=\lin(P-P)$ and $V=\lin(Q-Q)$. Set $r=\dim(U+V)$.

\begin{proposition}
The Cayley polytope of $P$ and $Q$ has dimension $r+1$. For every integer $n\geq0$, we have 
\begin{equation}\label{eq:cayley}
 L_{\Cay(P,Q)}(n)=
 \sum_{k=0}^{n}\bigl|((n-k)P+kQ)\cap\Z^d\bigr|.
\end{equation}
\end{proposition}
\begin{proof}
Let $p_0\in P\cap\Z^d$ and $q_0\in Q\cap\Z^d$.
The direction space of the Cayley polytope is generated by
$U\times\{0\}$, $V\times\{0\}$, and $(q_0-p_0,1)$.
The last vector is independent of the horizontal subspace, which has
dimension $r$. This proves the dimension formula.

Let $n\geq1$. By grouping together the coefficients of points in each layer in a convex combination, we show that the slice of
$n\Cay(P,Q)$ at real height $t\in[0,n]$ is
\[ \bigl((n-t)P+tQ\bigr)\times\{t\}.
\]
At $t=0$ or $t=n$, this means the corresponding single dilated layer, using $0P=0Q=\{0\}$. 
A lattice point must have an integer height $k\in\{0,\ldots,n\}$. 
By counting the disjoint slices, we obtain \eqref{eq:cayley}. At $n=0$, both sides equal one.
\end{proof}

Equation~\eqref{eq:cayley} expresses the Ehrhart polynomial through mixed Minkowski counts. 
Its $h^*$-polynomial is then obtained from \eqref{eq:conversion} with dimension $r+1$.
The formula differs from the ordinary convolution for a lattice join: an integer point of $(n-k)P+kQ$ need not split into an integer point of $(n-k)P$ and an integer point of $kQ$.

\subsection{Equivalent conditions for IDP and spanning}

\begin{theorem}\label{prop:cayley-idp}
The polytope $\Cay(P,Q)$ is IDP if and only if, for all nonnegative integers $a,b$, every lattice point of $aP+bQ$ is a sum of $a$ lattice points of $P$ and $b$ lattice points of $Q$. By convention, an empty sum evaluates to zero.
\end{theorem}
\begin{proof}
The slice description at dilation one shows that the lattice points of $\Cay(P,Q)$ are exactly
\[\bigl((P\cap\Z^d)\times\{0\}\bigr) \ \cup\ \bigl((Q\cap\Z^d)\times\{1\}\bigr).
\]
Suppose first that the Cayley polytope is IDP. Fix $a,b\geq0$ with $a+b>0$ and $u\in(aP+bQ)\cap\Z^d$. 
Assume $u=av_1+bv_2$ with $v_1\in P$ and $v_2\in Q$. Therefore, we have
$$(u,b)=(av_1,0)+(bv_2,b)=a(v_1,0)+b(v_2,1)= (a+b) \left( \frac{a}{a+b}(v_1, 0) + \frac{b}{a+b}(v_2, 1) \right).$$
The point $(u,b)$ belongs to $(a+b)\Cay(P,Q)$. 
An IDP decomposition has $a+b$ summands, each at
height zero or one. Their heights sum to $b$, so exactly $b$ summands
come from the $Q$ layer and the remaining $a$ come from the $P$ layer.
Projecting to $\R^d$ gives the required mixed decomposition.
The case $a=b=0$ is the empty sum.

Conversely, fix $n\geq1$ and a lattice point $(u,k)$ of
$n\Cay(P,Q)$. Its height satisfies $0\leq k\leq n$ and $u\in(n-k)P+kQ$. 

We apply the assumption with $a=n-k$ and $b=k$. 
Moreover, we lift the resulting summands of $P$ to height zero and those of $Q$ to height one. 
This yields exactly $n$ lifted summands that evaluate to $(u,k)$, thereby establishing the IDP.
\end{proof}

\cref{prop:cayley-idp} should be attributed to Tsuchiya; see \cite[Theorem 1.3]{Tsuchiya} for a more general case.
For completeness, we also provide a proof here.
In \cref{prop:cayley-idp}, the IDP of $P$ and $Q$ verifies only the cases $b=0$ and $a=0$ in this criterion. 
The cases with both integers positive impose additional conditions.

\begin{theorem}\label{prop:cayley-spanning}
Suppose that $P$ and $Q$ are spanning in their integer affine coordinates. Then $\Cay(P,Q)$ is spanning if and only if
$(\Z^d\cap U)+(\Z^d\cap V)=\Z^d\cap(U+V)$ holds. In particular, it is spanning when both input polytopes are full-dimensional and spanning in $\R^d$.
\end{theorem}
\begin{proof}
To determine whether $\Cay(P,Q)$ is spanning, we analyze the lattice generated by the differences of its integer points. 
Fix base points $p_0\in P\cap\Z^d$ and $q_0\in Q\cap\Z^d$. 
The integer points of the Cayley polytope reside entirely within the two layers $(P\cap\Z^d)\times\{0\}$ and $(Q\cap\Z^d)\times\{1\}$. 
Consequently, the lattice of differences is generated by the horizontal displacements within each individual layer, together with a single cross layer displacement. 
Specifically, this lattice is generated by the elements
\[ (u-p_0,0) \text{ for } u\in P\cap\Z^d, \qquad  (v-q_0,0) \text{ for } v\in Q\cap\Z^d, \qquad  \text{and } (q_0-p_0,1).
\]

We now consider the affine map $\Phi \colon \R^d\times\R \to \R^d\times\R$; it is defined by 
\[\Phi(x,t) = \bigl(x-p_0-t(q_0-p_0), t\bigr).
\]
The inverse of $\Phi$, given explicitly by $(x,t) \mapsto \bigl(x+p_0+t(q_0-p_0), t\bigr)$, is again an integral affine transformation. Consequently, $\Phi$ restricts to an automorphism of the ambient lattice $\Z^{d+1}$.
By construction, $\Phi$ translates our chosen base points $(p_0,0)$ and $(q_0,1)$ to the origin $(0,0)$ and the unit vector $(0,1)$, respectively. 

In this transformed coordinate system, the affine span of the Cayley polytope is mapped to $(U+V)\times\R$. 
The ambient lattice restricted to this affine span is exactly $\bigl(\Z^d\cap(U+V)\bigr)\times\Z$. 

We now determine the image of the difference lattice under the linear part of $\Phi$. By the hypothesis that $P$ and $Q$ are spanning, the horizontal displacements $u-p_0$ generate the sublattice $\Z^d\cap U$, and the displacements $v-q_0$ generate $\Z^d\cap V$. 
Together with the vertical generator mapping to $(0,1)$, the transformed difference lattice is precisely
\[
  \bigl((\Z^d\cap U)+(\Z^d\cap V)\bigr)\times\Z.
\]

By definition, the Cayley polytope is spanning if and only if this generated difference lattice coincides with the ambient lattice of its affine span. By comparing the two Cartesian products, this equality holds if and only if $(\Z^d\cap U)+(\Z^d\cap V) = \Z^d\cap(U+V)$. Finally, if both $P$ and $Q$ are full-dimensional in $\R^d$, then $U=V=\R^d$, rendering this geometric condition trivially satisfied.
\end{proof}

\subsection{Counterexamples}

\begin{theorem}\label{ex:cayley-reeve}
The Cayley polytopes of two lattice polytopes does not preserve the symmetry, unimodality, log-concavity, ultra log-concavity, real-rootedness, or $\gamma$-positivity of the $h^*$-polynomial; nor does it preserve spanning, very ampleness, IDP, unimodular triangulations, regular unimodular triangulations, or quadratic triangulations.
In addition, the Cayley polytope of two lattice polytopes also does not preserve Ehrhart positivity or magic positivity.
\end{theorem}
\begin{proof}
Let $m\geq2$, and take
\[ P=[(0,0),(1,0)],\qquad Q=[(0,0),(1,m)]\subset\R^2.
\]
Both segments are primitive in their integer affine coordinates. Each has Ehrhart polynomial $n+1$, $h^*$-polynomial $h_P^*(z)=h_Q^*(z)=1$, and a quadratic triangulation. Their Cayley polytope is the three-dimensional simplex
\[ R_m=\conv\bigl((0,0,0),(1,0,0),(0,0,1),(1,m,1)\bigr).
\]
Note that $R_m$ is the Reeve's tetrahedron, see \cite{Reeve} and \cite[Example 3.22]{BR}.

An integer height is zero or one, and both slices are primitive segments. Hence the four displayed vertices are all the lattice
points of $R_m$. The determinant of the three nonzero vertex columns has absolute value $m$. Thus $R_m$ has normalized volume $m$.
Its lattice-point difference group has index $m$ in $\Z^3$.

According to \cite[Example 3.22]{BR}, we know that 
\begin{align*}
 h_{R_m}^*(z)&=1+(m-1)z^2,\\
 L_{R_m}(n)&=\binom{n+3}{3}+(m-1)\binom{n+1}{3}=\frac m6n^3+n^2+\frac{12-m}{6}n+1.
\end{align*}
For every $m\geq2$, the $h^*$-polynomial has an internal zero and is not unimodal or log-concave. It is therefore neither ultra log-concave nor real-rooted. For $m\geq3$ it is not symmetric about its degree two. When $m=2$, it is symmetric, but
\[ 1+z^2=(1+z)^2-2z
\]
has a negative $\gamma$-coefficient. For $m=13$, the coefficient of $n$ in $L_{R_m}$ is $-1/6$, although the $P$ and $Q$ are magic positive. Thus every coefficient property listed in \cref{tab:coefficients} can fail under the Cayley construction.

The geometric property failures can also be checked directly in every large degree. 
Fix $1\leq j\leq m-1$. It is easy to obtain $(1,j,1)\in2R_m$. Since $0\in R_m$, one has $2R_m\subset nR_m$ for every $n\geq2$.
However, every lattice point of $R_m$ has second coordinate divisible by $m$. The same is true of every sum of such points, whereas the second coordinate of $(1,j,1)$ is not divisible by $m$.
Therefore $(1,j,1)$ is not a sum of $n$ lattice points of $R_m$ for any $n\geq2$. The polytope is not very ample or IDP and admits no unimodular, regular unimodular, or quadratic triangulation. 
Furthermore, it is nonspanning, as demonstrated by the preceding index computation.
\end{proof}

The failure of the very ample property can occur even for full-dimensional unimodular inputs.
\begin{example}
Let us recall that the four-dimensional simplices $A,D$ and the vector $v=(1,1,3,0)$ from \cref{ex:full-dimensional-holes}. 
Set $C=\Cay(A,D)\subset\R^5$. Its dimension is five. 
Since polytopes $A$ and $D$ are spanning and full-dimensional in $\R^4$, by \cref{prop:cayley-spanning}, we know that $C$ is spanning.

For every $n\geq3$, the integer point
\[ w=(1,1,1,0,1) \in nC
\]
Indeed, its horizontal coordinate is
\[ (1,1,1,0) =\left(\frac23e_1+\frac23e_2\right)+\frac13v\in(n-1)A+D,
\]
because the first summand belongs to $2A\subset(n-1)A$ and the second belongs to $D$. 
Equation~\eqref{eq:cayley} identifies this as the height-one slice of $nC$.

Suppose that $w$ is a sum of $n$ lattice points of $C$. Exactly one summand must lie in the $D$ layer, because the last coordinates sum to one. The fourth coordinates are nonnegative and sum to zero. Thus all summands lie on the fourth-coordinate-zero faces of their layers. The corresponding face of $A$ is $\conv(0,e_1,e_2)$, whose third coordinate is always zero. 
The corresponding face of $D$ is the primitive segment $[0,v]$. Its only lattice points are $0$ and $v$, since the first coordinate of $v$ is one. Its third coordinate is therefore zero or three. The sum cannot have third coordinate one.
This contradiction holds for every $n\geq3$. 

Therefore, $C$ is neither very ample nor IDP, and it possesses none of the three unimodular triangulation properties, even though it is spanning and generated by full-dimensional unimodular inputs.
\end{example}

\subsection{Cayley polytopes of boxes}

Positive preservation results can be established for specific families, provided the associated mixed sums exhibit additional combinatorial structure.
Fix integers $n\geq1$ and $d\geq0$, and a matrix $A=(a_{ij})\in\Z_{\geq0}^{n\times(d+1)}$ with no zero row. For $0\leq j\leq d$, define
\[ D_j=\prod_{i=1}^n[0,a_{ij}]\subset\R^n, \qquad v_0=0,\quad v_j=e_j\ (1\leq j\leq d),
\]
where $v_0,\ldots,v_d$ are the vertices of the standard simplex $\Delta_d\subset\R^d$. The associated Cayley polytope is
\[ \mathcal C(A)=\conv\left(\bigcup_{j=0}^dD_j\times\{v_j\}\right) \subset\R^{n+d}.
\]
Since $A$ is permitted to contain zero entries, the individual boxes may be of lower dimension. 
In the trivial case $d=0$, the construction naturally reduces to the single box $D_0$.

A full-dimensional lattice polytope of dimension $N$ is called \emph{smooth} if exactly $N$ edges meet at each vertex and the
primitive integer vectors along those edges form a basis of $\Z^N$. Here the primitive vector along an edge is the first
nonzero lattice vector on the corresponding ray from its vertex.

\begin{theorem}{\em (Ferroni; \cite{Ferroni2026})}
For the matrix $A$ and the Cayley polytope $\mathcal C(A)$ defined above, the following statements hold.
\begin{enumerate}
\item The dimension is $n+d$, and $\mathcal C(A)$ has a regular unimodular flag triangulation \emph{\cite[Propositions~2.1--2.2]{Ferroni2026}}.
\item For every integer $k\geq0$, its Ehrhart polynomial satisfies the following formula \emph{\cite[Theorem~3.1]{Ferroni2026}}:
\begin{equation}\label{eq:box-cayley}
 L_{\mathcal C(A)}(k)= \sum_{\substack{b_0+\cdots+b_d=k\\b_0,\ldots,b_d\in\Z_{\geq0}}} \prod_{i=1}^n\left(1+\sum_{j=0}^d a_{ij}b_j\right).
\end{equation}
\item If every entry of $A$ is positive, then $\mathcal C(A)$ is smooth \emph{\cite[Proposition~2.3]{Ferroni2026}}.
\item If every row of $A$ sums to two, then $\mathcal C(A)$ is Gorenstein of index $d+1$ \emph{\cite[Proposition~2.4]{Ferroni2026}}.
\end{enumerate}
\end{theorem}

The triangulation assertion is an application of the integral chimney construction, which adjoins a coordinate between integral affine lower and upper bounds on the preceding polytope. The relevant triangulation results are due to Haase--Paffenholz--Piechnik--Santos \cite[Theorem~2.8 and Corollary~2.9]{HPPS}. 

Recently, by using the Cayley polytopes of boxes, Ferroni \cite[Theorem~1.2]{Ferroni2026} proved that IDP and even a quadratic triangulation do not imply the unimodality of $h^*$-polynomial.
Based on the construction of Ferroni, we may perhaps make the following remark.

\begin{remark}
There are lattice polytopes $U,D$ with real-rooted $h^*$-polynomials such that $\Cay(U,D)$ has a quadratic triangulation and a nonunimodal $h^*$-polynomial.
\end{remark}

\section{Reflexive polarity}\label{sec:polarity}

Throughout this section, $P\subset\R^d$ is full-dimensional, $0\in\relint(P)$, and polarity uses the Euclidean scalar product:
\[ P^*=\{y\in\R^d:\langle y,x\rangle\leq1\text{ for every }x\in P\}.
\]
Lattice points of both $P$ and $P^*$ are counted in $\Z^d$. The bipolar identity gives $(P^*)^*=P$. 
Consequently, if $P$ is reflexive, then $P^*$ is a reflexive lattice polytope of dimension $d$.
Hibi's symmetry theorem \cite{Hibi} says that both $h_P^*(z)$ and $h_{P^*}^*(z)$ are symmetric of degree $d$. 
Thus polarity preserves symmetry on the class of reflexive polytopes. 

In dimension zero the only polytope is $\{0\}$, and polarity fixes it.
For full-dimensional polytopes $P\subset\R^p$ and $Q\subset\R^q$ with interior origins, the standard identities exchanging products and free sums are
\begin{equation}\label{eq:polar-products}
 (P\times Q)^*=P^*\oplus Q^*,\qquad (P\oplus Q)^*=P^*\times Q^*;
\end{equation}
see \cite{Stapledon}. 
These identities use the usual product coordinates and the lattice $\Z^p\times\Z^q$.

\subsection{Positivity}

\begin{theorem}\label{prop:polar-positive}
There is an Ehrhart positive reflexive lattice polytope whose polar has a negative linear Ehrhart coefficient.
\end{theorem}
\begin{proof}
The standard cross-polytope
\[C_d=\conv(\pm e_1,\ldots,\pm e_d)
\]
is Ehrhart positive for every integer $d\geq1$; see \cite[Section~2.2.1]{Liu}. 
Its polar is the cube $[-1,1]^d$, as follows from the inequalities $|y_i|\leq1$ imposed by its vertices. 
In particular, both polytopes in this polar pair are reflexive.
Set
\[P=C_{10}\times[-1,1]\subset\R^{11}.
\]
By \eqref{eq:polar-products}, we have 
\[P^*=[-1,1]^{10}\oplus[-1,1]=B_{10},
\]
where $B_{10}$ is the cube bipyramid of \cref{ex:bipyramid}.
Both $P$ and $P^*$ have integer vertices and interior origins, so $P$ is reflexive. Furthermore, we have 
\[L_P(n)=(2n+1)L_{C_{10}}(n).
\]
Therefore, $P$ is Ehrhart positive since every coefficient of $L_{C_{10}}$ is positive.
On the other hand, by \eqref{eq:bipyramid-polynomial}, we have the linear coefficient of $L_{P^*}$ as $-4450/33$. 
This proves the assertion.
\end{proof}

\begin{theorem}\label{ex:polar-triangle}
The reflexive polarity preserves none of real-rootedness, ultra log-concavity, $\gamma$-positivity of the $h^*$-polynomial, and magic positivity.
\end{theorem}
\begin{proof}
We consider the polytope
\[S=\conv(e_1,e_2,-e_1-e_2).
\]
The origin is the average of its three vertices and lies in its interior.
By the definition of polarity, we have 
\[S^*=\{(y_1,y_2):y_1\leq1,\ y_2\leq1,\ -y_1-y_2\leq1\}=\conv\bigl((1,1),(1,-2),(-2,1)\bigr).
\]
Indeed, intersecting any two of the three boundary lines gives one of the displayed vertices, and each satisfies the remaining inequality. At the origin, all three left sides are less than one. Thus $S$ and $S^*$ are reflexive triangles.

According to the Pick's theorem \cite[Theorem 2.8]{BR}, we obtain the Ehrhart polynomials and the $h^*$-polynomials in the following table.
\[\begin{array}{c|c|c|c}
 &L(n)&h^*(z)&\text{magic coefficients}\\ \hline
 S&\frac32n^2+\frac32n+1&1+z+z^2&(1,-\frac12,1)\\
 S^*&\frac92n^2+\frac92n+1&1+7z+z^2&(1,\frac52,1)
 \end{array}
\]
Here the magic coefficients use the ordered basis $(n+1)^2,n(n+1),n^2$. 

Take $S^*$ as the input polytope to polarity. Its numerator $1+7z+z^2$ has discriminant $45>0$, belongs to $\ULC(2)$ because $7^2\geq4$, and has the expansion
\[1+7z+z^2=(1+z)^2+5z.
\]
It is therefore real-rooted and $\gamma$-positive, and its displayed magic coefficients are nonnegative. 

Its polar $S$ has $h^*$-polynomial $1+z+z^2$, whose discriminant is $-3$. This numerator fails $\ULC(2)$ because $1^2<4$, and its expansion
\[1+z+z^2=(1+z)^2-z
\]
has a negative $\gamma$-coefficient. The negative middle magic coefficient of $L_S(n)$ also rules out magic positivity. Thus reflexive polarity preserves none of real-rootedness, ultra log-concavity, $\gamma$-positivity, and magic positivity.
\end{proof}

\subsection{Coefficient properties and geometric properties}

\begin{theorem}\label{ex:polar-payne}
The reflexive polarity preserves neither unimodality nor log-concavity of the $h^*$-polynomial.
Furthermore, it preserves none of very ampleness, IDP, unimodular triangulations, regular unimodular triangulations, or quadratic triangulations.
\end{theorem}
\begin{proof}
We consider the six-dimensional simplex
\[ R=\conv\bigl(e_1,\ldots,e_6,-(1,1,1,1,1,3)\bigr).
\]
This is the member with parameters $r=0$, $s=3$, and $k=2$ of Payne's family \cite{Payne}, in the coordinates recorded by
Liu \cite[Example~2.2.5]{Liu}. Its $h^*$-polynomial is given by 
\[ h_R^*(z)=(1+z+z^2)(1+z^2+z^4)=1+z+2z^2+z^3+2z^4+z^5+z^6.
\]
It is neither unimodal nor log-concave. We will verify that $R^*$ has a log-concave $h^*$-polynomial and a quadratic
triangulation, while $R$ is not very ample.

First, the identity
\[
 0=\frac19\bigl(-(1,1,1,1,1,3)+e_1+\cdots+e_5+3e_6\bigr)
\]
expresses the origin as a convex combination of all seven vertices with positive coefficients. These vertices affinely span $\R^6$, so $0\in\relint(R)$. The polar inequalities are
\[ y_i\leq1\quad(1\leq i\leq6),\qquad y_1+\cdots+y_5+3y_6\geq-1.
\]
The integer affine change $u_i=1-y_i$ therefore gives $R^*=\one-T$, where
$$T=\{u\in\R_{\geq0}^6:u_1+\cdots+u_5+3u_6\leq9\}=\conv(0,9e_1,\ldots,9e_5,3e_6).$$
In particular, $R^*$ has integer vertices, which directly verifies reflexivity. For each integer $n\geq0$, the map
$u\longmapsto n\one-u$ is a bijection from $nT\cap\Z^6$ to $nR^*\cap\Z^6$. Thus the reflection and translation preserve the
Ehrhart polynomial.

For a lattice point of $nT$, its sixth coordinate is an integer $j\in\{0,\ldots,3n\}$. Once this coordinate is fixed, the other
five are nonnegative integers with sum at most $9n-3j$. Adding a nonnegative slack coordinate identifies these points with
six-tuples of nonnegative integers having sum $9n-3j$. Their number is $\binom{9n-3j+5}{5}$. Consequently, we have 
\[ L_{R^*}(n)=\sum_{j=0}^{3n}\binom{9n-3j+5}{5} \qquad(n\geq0).
\]
By \eqref{eq:conversion}, we obtain
\[ h_{R^*}^*(z)=1+2514z+40236z^2+91645z^3+40236z^4+2514z^5+z^6.
\]
Therefore, the polynomial $h_{R^*}^*(z)$ is symmetric, unimodal, and log-concave.
Thus $R^*$ is an input polytope whose $h^*$-polytope is both unimodal and log-concave, whereas its polar $R$ has neither property.

We next verify the triangulation assertion. A \emph{Nakajima polytope} is obtained from a lattice point by repeatedly replacing a lattice polytope $Q$ with
\[\{(x,t):x\in Q,\ 0\leq t\leq f(x)\},
\]
where $f$ is an integral affine function that is nonnegative on $Q$. Here integral affine means that the constant term and all linear coefficients are integers in the chosen lattice coordinates. 
We know that every Nakajima polytope admits a regular unimodular flag triangulation \cite[Corollary~2.9]{HPPS}.
This is a quadratic triangulation in our terminology. This result allows the upper function to vanish on a face.

For $u\in T$, define $t_i=u_1+\cdots+u_i$ for $1\leq i\leq5$ and retain the coordinate $u_6$. 
This transformation is an integer linear bijection of $\Z^6$: its inverse is
\[ u_1=t_1,\qquad u_i=t_i-t_{i-1}\quad(2\leq i\leq5).
\]
Its matrix on the first five coordinates is triangular with diagonal entries one. Thus it is unimodular. 
In this coordinates, $T$ is described by the inequalities
\[
  0 \leq u_6 \leq 3, \qquad 0 \leq t_5 \leq 9-3u_6, \qquad 0 \leq t_i \leq t_{i+1} \quad (1 \leq i \leq 4).
\]
Starting from a point, we successively adjoin the coordinates in the order $u_6, t_5, t_4, t_3, t_2, t_1$. The corresponding upper functions are given by
\[
  3, \quad 9-3u_6, \quad t_5, \quad t_4, \quad t_3, \quad t_2.
\]
Each of these functions is integral affine and nonnegative on the previously constructed polytope.
This proves that the transformed polytope is a Nakajima polytope. 

Therefore, this construction yields a quadratic triangulation, which is subsequently mapped to a quadratic triangulation of $R^*$ via the inverse unimodular transformation and the integral affine map $u \mapsto \one - u$.
Indeed, an integer affine lattice bijection preserves normalized volumes and face incidences, and transporting the lifting heights
preserves regularity. In particular, $R^*$ has a regular unimodular triangulation, a unimodular triangulation, IDP, very ampleness, and spanning.

Finally, by $h_1^*(R)=1$ and \eqref{eq:conversion}, we have $|R\cap\Z^6|=6+1+1=8$. 
Its seven vertices and the origin already give eight distinct lattice points. This is the complete list. 
The equality
\[ -e_6=\frac13\bigl(e_1+\cdots+e_5-(1,1,1,1,1,3)\bigr)
\]
expresses $-e_6$ as a nonnegative combination of vertices of $R$ whose coefficients sum to two. Thus $-e_6\in2R$.
For every integer $n\geq2$, by adding $n-2$ copies of $e_1$, we have
\[
 (n-2)e_1-e_6\in nR\cap\Z^6.
\]
Suppose this point were a sum of $n$ lattice points of $R$. Let $a$ denote the number of summands equal to $-(1,1,1,1,1,3)$.
The sixth coordinate forces the number of $e_6$ summands to be $3a-1$, so $a\geq1$. The second through fifth coordinates force
$a$ copies of each of $e_2,\ldots,e_5$, and the first coordinate forces $n-2+a$ copies of $e_1$. The number of nonzero summands is
therefore
\[a+(3a-1)+4a+(n-2+a)=n-3+9a\geq n+6>n,
\]
a contradiction. There is consequently an undecomposable lattice point in every $nR$ with $n\geq2$. Hence $R$ is not very ample and is not IDP. Since a unimodular triangulation implies IDP, $R$ has none of the three unimodular triangulation properties in \cref{tab:coefficients}. It is nevertheless spanning, because $0,e_1,\ldots,e_6\in R$.
This proves the five geometric failures in the direction $R^*\mapsto R$.
\end{proof}

\begin{theorem}\label{ex:polar-spanning}
The reflexive polarity does not preserve spanning.
\end{theorem}
\begin{proof}
In the lattice $\Z^4$, we consider the polytope 
\[ V=\conv\bigl(\pm e_1,\pm e_2,\pm e_3,\pm(-1,-1,-1,2)\bigr).
\]
The matrix with columns $e_1,e_2,e_3,(-1,-1,-1,2)$ has determinant two. These directions are linearly independent, so their centrally symmetric convex hull has dimension four and contains the origin in its interior. The polar is
\[ V^*=\{y\in\R^4:|y_1|,|y_2|,|y_3|\leq1,\ |-y_1-y_2-y_3+2y_4|\leq1\}.
\]
The invertible linear map $y\mapsto(y_1,y_2,y_3,-y_1-y_2-y_3+2y_4)$ sends $V^*$ to $[-1,1]^4$. Its sixteen vertices are therefore exactly
\[\left(\varepsilon_1,\varepsilon_2,\varepsilon_3,\frac{\varepsilon_1+\varepsilon_2+\varepsilon_3+\varepsilon_4}{2}\right),
 \qquad \varepsilon_i\in\{-1,1\}.
\]
The sum of four odd integers is even, so the last coordinate is integral. Thus $V$ and $V^*$ are reflexive lattice polytopes.

To determine all lattice points of $V$, use the rational linear map
\[T(x)=\left(x_1+\frac{x_4}{2},x_2+\frac{x_4}{2},x_3+\frac{x_4}{2},\frac{x_4}{2}\right).
\]
It sends the four positive directions defining $V$ to $e_1,e_2,e_3,e_4$, respectively, and hence
\[ T(V)=\{t\in\R^4:|t_1|+|t_2|+|t_3|+|t_4|\leq1\}.
\]
Its inverse is
\[ T^{-1}(t)=(t_1-t_4,t_2-t_4,t_3-t_4,2t_4).
\]
Thus, the transformed lattice is exactly
\[ T(\Z^4)=\Z^4\ \cup\ \bigl(\Z^4+\frac{1}{2}\one\bigr).
\]
Indeed, even $x_4$ gives four integer coordinates and odd $x_4$ gives four half-integer coordinates. 
The inverse formula shows that every point in either displayed coset comes from $\Z^4$.
Every point in the half-integer coset has each coordinate of absolute value at least $\frac{1}{2}$. 
Therefore, its sum of absolute coordinates is at least two. No such point belongs to $T(V)$. 
In the integer coset, the only points whose sum of absolute coordinates is at most one are $0,\pm e_1,\ldots,\pm e_4$. 
Their inverse images are the origin and the eight vertices of $V$.

Because $0\in V$, its lattice-point difference group is generated by $e_1,e_2,e_3,(-1,-1,-1,2)$. 
Every element has even fourth coordinate, and
\[2e_4=(-1,-1,-1,2)+e_1+e_2+e_3.
\]
Thus the group is exactly $\Z^3\times2\Z$, which has index two in $\Z^4$. It follows that $V$ is not spanning.
In contrast, direct substitution into the inequalities for $V^*$ shows that it contains
$0$, $e_1$, $e_2$, $e_3$, and $(1,0,0,1)$.
The differences of these points generate $\Z^4$, since $(1,0,0,1)-e_1=e_4$. Hence $V^*$ is spanning, and its polar
$V$ is not. This proves failure of spanning preservation.
\end{proof}

\section{Integral dilations}\label{sec:dilation}

For a positive integer $r$, the \emph{integral dilation} of a lattice $d$-polytope $P\subset\R^d$ is
$rP=\{rx:x\in P\}$. It has the same dimension as $P$. The dilation of the set $P$ should be distinguished from scaling
its lattice points: the inclusion $r(P\cap\Z^d)\subseteq rP\cap\Z^d$ is usually strict.
For example, $2[0,1]$ contains the lattice point $1$.

We use $c(P)=r_P=d+1-s_P$ in this section to avoid confusing the codegree with the dilation factor $r$.

\subsection{Ehrhart positivity and magic positivity}

For every nonnegative integer $n$ we have $n(rP)=(rn)P$, and therefore $L_{rP}(n)=L_P(rn)$.
If $L_P(n)=\sum_{j=0}^d c_jn^j$, this becomes
\[ L_{rP}(n)=\sum_{j=0}^d r^jc_jn^j.
\]
Therefore, $rP$ is Ehrhart positive if and only if $P$ is Ehrhart positive.

\begin{theorem}
Integral dilations preserve magic positivity.
\end{theorem}
\begin{proof}
Suppose $L_P(n)=\sum_{j=0}^d a_jn^j(n+1)^{d-j}$ with $a_j\geq0$. For each $j$, we have 
\begin{align*}
(rn)^j(rn+1)^{d-j}=r^jn^j\bigl((n+1)+(r-1)n\bigr)^{d-j}
=\sum_{k=j}^d r^j\binom{d-j}{k-j}(r-1)^{k-j}n^k(n+1)^{d-k}.
\end{align*}
By summing over $j$, we obtain
\[L_{rP}(n)=\sum_{k=0}^d\left(\sum_{j=0}^ka_jr^j\binom{d-j}{k-j}(r-1)^{k-j}\right)n^k(n+1)^{d-k}.
\]
Every coefficient in parentheses is nonnegative because
$r\geq1$. This is the required expansion in the degree-$d$
magic basis.
\end{proof}

\subsection{Coefficient properties of $h^*$-polynomial}\label{Subsection-veronese}

For a characterization of the $h_{rP}^*$-polynomial, we need the following definitions and notation.
For a polynomial or formal power series $f(z)=\sum_{m\geq0}b_mz^m$, define
\[f^{\langle r,0\rangle}(z)=\sum_{n\geq0}b_{rn}z^n
 \qquad \text{and}\qquad B_r(z)=1+z+\cdots+z^{r-1}.
\]
Thus extraction changes the exponents $rn$ to $n$.

We define the zero-offset \emph{Veronese operator} by
\[ U_r^{d+1}h(z) = \bigl(B_r(z)^{d+1}h(z)\bigr)^{\langle r,0\rangle}.
\]
Note that the superscript $d+1$ corresponds to the exponent of the denominator in the rational generating function, where $d$ denotes the dimension of the underlying polytope.
For some results on the Veronese operator, we refer to \cite{Beck-Stapledon,BrentiWelker}.

\begin{proposition}\label{prop:dilation-formula}
For every integer $r\geq1$, we have 
\begin{equation}\label{eq:dilation-extraction}
 h_{rP}^*(z)=U_r^{d+1}h_P^*(z)
 =\bigl(B_r(z)^{d+1}h_P^*(z)\bigr)^{\langle r,0\rangle}.
\end{equation}
Moreover,
\begin{equation}\label{eq:dilation-degree}
 c(rP)=\left\lceil\frac{c(P)}r\right\rceil,
 \qquad
 \deg h_{rP}^*=d+1-\left\lceil\frac{d+1-s_P}r\right\rceil.
\end{equation}
\end{proposition}
\begin{proof}
For \eqref{eq:dilation-extraction}, this is the standard Veronese formula in \cite[Lemma~3.1]{Jochemko} and
\cite[Section~1]{Zhang}.
For the codegree assertion, set $m\in\operatorname{int}(c(P)P)\cap\Z^d$ and $v\in P\cap\Z^d$. 
If $k\geq c(P)$, then
\[
 m+(k-c(P))v\in\operatorname{int}(kP)\cap\Z^d.
\]
Indeed, since $m$ belongs to the interior of $c(P)P$, there exists an open ball centered at $m$ that is entirely contained within $c(P)P$. Translating this ball by the vector $(k-c(P))v \in (k-c(P))P$ yields an open neighborhood of $m+(k-c(P))v$ that lies within the Minkowski sum $c(P)P + (k-c(P))P = kP$.
Thus $kP$ has an interior lattice point for every integer $k\geq c(P)$, and none for $1\leq k<c(P)$ by the definition of codegree.
It follows that $n(rP)$ has an interior lattice point exactly when $rn\geq c(P)$. The smallest positive integer $n$ with this property is $\lceil c(P)/r\rceil$. The degree formula now follows from $\deg h_{rP}^*=d+1-c(rP)$.
\end{proof}

By \eqref{eq:conversion}, we also have 
\[h_i^*(rP)=\sum_{j=0}^i(-1)^j\binom{d+1}{j}\sum_{k=0}^{d}h_k^*(P) \binom{r(i-j)+d-k}{d},\qquad 0\leq i\leq d.
\]

\begin{lemma}\label{lem:uniform-unimodal}
If $h(z)$ is a polynomial with nonnegative and unimodal coefficients, then for any integer $r \ge 1$, the product $B_r(z)h(z)$ is unimodal.
\end{lemma}
\begin{proof}
The assertion is trivial for $r=1$. For $r \ge 2$, let $h(z) = \sum a_i z^i$ by setting $a_i=0$ outside its support.  
Suppose the coefficient sequence achieves a peak at $m$. 
The $i$-th coefficient of the product $B_r(z)h(z)$ is given by $b_i = \sum_{j=0}^{r-1} a_{i-j}$, which yields the difference
$$b_i - b_{i-1} = a_i - a_{i-r}.$$

To analyze the difference $b_i - b_{i-1} = a_i - a_{i-r}$, we partition the indices into three distinct ranges. 
For $i \le m$, both $i$ and $i-r$ lie within the nondecreasing portion of the sequence (or precede its support), which yields $b_i - b_{i-1} \ge 0$. 
Conversely, for $i \ge m+r$, both indices fall within the nonincreasing portion (or beyond the support), ensuring $b_i - b_{i-1} \le 0$. 
In the intermediate range $m < i < m+r$, the term $a_i$ is nonincreasing with respect to $i$, whereas $a_{i-r}$ is nondecreasing. Consequently, the difference $a_i - a_{i-r}$ is monotonically nonincreasing across this interval.

It follows that the sequence of first differences $(b_i - b_{i-1})$ undergoes at most one sign change, passing exclusively from nonnegative to nonpositive. This establishes the unimodality of the product.
\end{proof}

We also use the classical convolution theorem: if two finite nonnegative sequences are log-concave and have no internal zeros, then their convolution is also log-concave and has no internal zeros; see \cite{StanleyLC}. 
In polynomial language, their product is log-concave and has no internal zeros. The coefficient sequence of $B_r(z)$ consists of $r$ consecutive ones and satisfies these hypotheses.

\begin{theorem}\label{thm:dilation-unimodal-lc}
For every integer $r\geq1$, integral dilation preserves unimodality and log-concavity of $h^*$-polynomials, with the no-internal-zeros convention for log-concavity.
\end{theorem}
\begin{proof}
Suppose first that $h_P^*(z)$ is unimodal. Repeated application of \cref{lem:uniform-unimodal} shows that $B_r^{d+1}h_P^*(z)$ is unimodal. It is clear that taking the coefficients at indices $0,r,2r,\ldots$ preserves unimodality.
According to \eqref{eq:dilation-extraction}, we prove the unimodality of $h_{rP}^*(z)$.

Now suppose that $h_P^*(z)$ is log-concave with no internal zeros. The convolution theorem~\cite{StanleyLC}, applied successively to the $d+1$ factors $B_r(z)$, shows that $B_r^{d+1}(z)h_P^*(z)$ has a log-concave coefficient sequence $b_0,\ldots,b_N$ with every coefficient positive. 
Log-concavity is equivalent to the ratios $b_i/b_{i-1}$ being nonincreasing. Whenever all three selected indices occur, we have
\[\frac{b_{ri}}{b_{r(i-1)}}=\prod_{j=1}^r\frac{b_{r(i-1)+j}}{b_{r(i-1)+j-1}}\geq\prod_{j=1}^r\frac{b_{ri+j}}{b_{ri+j-1}}
 =\frac{b_{r(i+1)}}{b_{ri}}.
\]
Therefore, we have  $b_{ri}^2\geq b_{r(i-1)}b_{r(i+1)}$.
By \eqref{eq:dilation-extraction}, we prove the log-concavity assertion.
\end{proof}

In order to prove the result concerning real-rootedness, we need the following lemma.

\begin{lemma}{\em (Zhang; \cite[Corollary~3.4, Corollary~3.3]{Zhang})}\label{thm:zhang-veronese}
Let $h(z)$ be a nonzero polynomial with nonnegative coefficients whose zeros are all negative and real. For any positive integers $r,N$, the polynomial
\[ \bigl(B_r(z)^N h(z)\bigr)^{\langle r,0\rangle}
\]
is real-rooted. Furthermore, the same conclusion holds when $h(z)$ is log-concave with no internal zeros and $\deg h(z)<2r$.
\end{lemma}

\begin{corollary}\label{cor:dilation-rr}
If $h_P^*(z)$ is real-rooted, then $h_{rP}^*(z)$ is real-rooted for every integer $r\geq1$.
If $h_P^*(z)$ is log-concave with no internal zeros and $s_P<2r$, then $h_{rP}^*(z)$ is real-rooted.
\end{corollary}
\begin{proof}
According to \eqref{eq:dilation-extraction} and \cref{thm:zhang-veronese} with $N=d+1$, we prove the first assertion.
The second assertion follows from the second assertion of \cref{thm:zhang-veronese} with the same value of $N$.
\end{proof}

For completeness, we mention the following result of Jochemko

\begin{proposition}{\em (Jochemko \cite[Corollary~4.4]{Jochemko})}\label{Theorem-Jochem}
For a lattice $d$-polytope $P$ with $s_P=\deg h_P^*$, the polynomial $h_{rP}^*(z)$ has only negative real zeros for every
positive integer $r\geq s_P$. All its zeros are distinct when $r\geq\min\{s_P+1,d\}$.
\end{proposition}

\cref{Theorem-Jochem} demonstrates that arbitrary polynomials eventually acquire real-rootedness under sufficiently large dilations. In contrast, \cref{cor:dilation-rr} establishes that if the initial polynomial $h_P^*(z)$ is already real-rooted, this property is preserved for all positive integer dilations.

\begin{example}
Integral dilations do not preserve the symmetry and $\gamma$-positivity of the $h^*$-polynomial.
The unit interval $P=[0,1]$ has $h_P^*(z)=1$. For $r=3$, we have $L_{3P}(n)=3n+1$ and $h_{3P}^*(z)=1+2z$.
This is not symmetric in its actual degree and is not $\gamma$-positive. 
\end{example}

\subsection{Ultra log-concavity}

We know that any real-rooted polynomial with nonnegative coefficients belongs to $\ULC(s)$. Consequently, \cref{cor:dilation-rr} ensures that every dilation of a real-rooted polynomial satisfies the ultra log-concave condition with respect to its actual degree. Furthermore, the corollary yields this same actual-degree ultra log-concave property for log-concave polynomials, provided that $s_P < 2r$. For arbitrary lattice $d$-polytopes, Jochemko's theorem (\cref{Theorem-Jochem}) guarantees this conclusion whenever the dilation factor satisfies $r \ge s_P$.

\begin{lemma}\label{lem:even-ulc}
Suppose $f(z) = \sum_{i=0}^s b_i z^i \in \ULC(N)$ with $b_0 > 0$, and adopt the convention that $b_i = 0$ for $i > s$. Then the extracted series $f^{\langle 2,0 \rangle}(z)$ belongs to $\ULC(\lfloor N/2 \rfloor)$.
\end{lemma}
\begin{proof}
Let $M = \lfloor N/2 \rfloor$. The case $M=0$ yields a constant polynomial, making the assertion trivial. For $M \ge 1$, we can decompose the normalized coefficients of $f^{\langle 2,0 \rangle}(z)$ into two factors:
\[\frac{b_{2i}}{\binom{M}{i}}
 = \frac{b_{2i}}{\binom{N}{2i}}\frac{\binom{N}{2i}}{\binom{M}{i}}.
\]
The first factor, viewed as a sequence in $i$, forms a subsequence of the log-concave sequence $b_j / \binom{N}{j}$. By the same ratio argument applied in the proof of \cref{thm:dilation-unimodal-lc}, this subsequence preserves log-concavity.

The second factor is positive for $0 \le i \le M$. For $1 \le i \le M$, the ratio of consecutive terms evaluates to
\[\frac{\binom{N}{2i}/\binom{M}{i}}{\binom{N}{2i-2}/\binom{M}{i-1}}
 = \frac{(N-2i+2)(N-2i+1)}{2(2i-1)(M-i+1)}
 = \begin{cases}
    \dfrac{2M-2i+1}{2i-1}, & \text{if } N=2M,\\[6pt]
    \dfrac{2M-2i+3}{2i-1}, & \text{if } N=2M+1.
   \end{cases}
\]
In either case, as $i$ increases, the numerator is decreasing while the positive denominator is increasing. Thus, the ratio is monotonically nonincreasing, establishing the log-concavity of the second factor.

Since the termwise product of two nonnegative log-concave sequences is inherently log-concave, the normalized coefficient sequence of $f^{\langle 2,0 \rangle}(z)$ is log-concave. Furthermore, because $f(z)$ possesses no internal zeros and $b_0 > 0$, 
we have $f^{\langle 2,0 \rangle}(z) \in \ULC(M)$.
\end{proof}

We use Liggett's convolution theorem \cite{Liggett}, which asserts that if $f \in \ULC(A)$ and $g \in \ULC(B)$, then their product satisfies $fg \in \ULC(A+B)$. 
Note that the designated orders $A$ and $B$ are permitted to exceed the actual degrees of the respective polynomials. 
As a fundamental application, the binomial polynomial $(1+z)^{d+1}$ belongs to $\ULC(d+1)$ since its normalized coefficients are one.

\begin{proposition}\label{prop:dilation-double-ulc}
The operation of doubling a polytope preserves ultra log-concavity with respect to both the dimension and the actual degree. Consequently, this preservation extends to any dilation factor of the form $r=2^k$, where $k\geq 0$.
\end{proposition}
\begin{proof}
Suppose that $h_P^*(z) \in \ULC(d)$. By Liggett's convolution theorem, we have
\[(1+z)^{d+1}h_P^*(z)\in\ULC(2d+1).
\]
By \cref{lem:even-ulc} with $N=2d+1$ and \eqref{eq:dilation-extraction} with $r=2$, we obtian $h_{2P}^*(z) \in \ULC(d)$.

Now, assume the stronger condition that $h_P^*(z) \in \ULC(s_P)$. In this case, Liggett's theorem implies that the product resides in $\ULC(s_P+d+1)$. A subsequent application of the same \cref{lem:even-ulc} deduces that
\[ h_{2P}^*(z) \in \ULC\left(\left\lfloor \frac{s_P+d+1}{2} \right\rfloor\right).
\]
According to the degree formula \eqref{eq:dilation-degree}, we find
\[
 \deg h_{2P}^* = d+1-\left\lceil\frac{d+1-s_P}{2}\right\rceil = \left\lfloor\frac{s_P+d+1}{2}\right\rfloor.
\]
Observe that this derived order coincides precisely with the new actual degree of $h_{2P}^*(z)$.

To establish the assertion for the factor $r=2^k$, we iterate the doubling procedure successively across the sequence of polytopes $P, 2P, \ldots, 2^{k-1}P$. The dimension $d$ remains invariant throughout this process, while the actual-degree preservation applies dynamically to the updated degree at each iterative step. This completes the proof.
\end{proof}

The partial results on ultra log-concavity above are relatively easy. Below we give the complete ultra log-concavity preservation property.

\begin{theorem}\label{thm:dilation-ulc-general}
Let $P\subset\R^d$ be a lattice $d$-polytope and set $s_P=\deg h_P^*(z)$. For every positive integer $r$, the following results hold.
\begin{enumerate}
\item If $h_P^*\in\ULC(d)$, then $h_{rP}^*\in\ULC(d)$.
\item If $h_P^*\in\ULC(s_P)$, then $h_{rP}^*\in\ULC(s_{rP})$, where $s_{rP}$ is the actual degree of $h_{rP}^*(z)$, and 
\[s_{rP}=\left\lfloor\frac{s_P+(r-1)(d+1)}r\right\rfloor=d+1-\left\lceil\frac{d+1-s_P}{r}\right\rceil.
\]
\end{enumerate}
\end{theorem}

The proof of this theorem is somewhat lengthy, and we will give it in \cref{Append-Section}.

\subsection{Geometric properties}

\begin{theorem}
Integral dilations preserve spanning, IDP, and very ampleness.
\end{theorem}
\begin{proof}
Given a $v\in P\cap\Z^d$. For any $x\in P$, we have 
\[(r-1)v+x=r\left(\frac{r-1}{r}v+\frac1r x\right)\in rP.
\]
Therefore, we have $(r-1)v+P\subseteq rP$.
Thus $rv$ and $(r-1)v+x$ are lattice points of $rP$ whenever $x\in P\cap\Z^d$. Their difference is $x-v$.
If these differences generate $\Z^d$ for $P$, they also do so for $rP$. Thus integral dilations preserve spanning.

Now suppose $P$ is IDP. Set $u\in n(rP)\cap\Z^d=(nr)P\cap\Z^d$.
Write $u=u_1+\cdots+u_{nr}$ with $u_i\in P\cap\Z^d$.
Partition the $nr$ summands into $n$ groups of $r$.
Each group sum is a lattice point of $rP$, because the
average of the $r$ points belongs to $P$ by convexity.
Their $n$ group sums decompose $u$ as required.

Finally, suppose $P$ is very ample. We choose an integer $N$ such that every lattice point of $mP$ decomposes into $m$
lattice points of $P$ whenever $m\geq N$. For every integer $n\geq\lceil N/r\rceil$, the preceding argument applies with $m=nr\geq N$. Hence every lattice point of $n(rP)$ decomposes into $n$ lattice points of $rP$, and $rP$ is very ample.
\end{proof}

We also note that Bruns--Gubeladze--Trung \cite[Corollary~1.3.3]{BGT} gave a sufficient condition for $rP$ to satisfy the integer decomposition property for any lattice polytope $P$.

\begin{theorem}{\em (Bruns--Gubeladze--Trung \cite[Corollary~1.3.3]{BGT})}
Let $P$ be a $d$-dimensional lattice polytope. Then, for any positive integer $r\geq d-1$, the dilation $rP$ is IDP.
\end{theorem}

The preservation property concerning the triangulation of $rP$ should be attributed to Haase, Paffenholz, Piechnik, and Santos.

\begin{theorem}{\em (Haase--Paffenholz--Piechnik--Santos; \cite[Theorems~4.8 and 4.11, Corollary~4.12 ]{HPPS})}\label{thm:dilation-triangulations}
For every integer $r\geq1$, the following result holds.
\begin{enumerate}
\item if $P$ has a unimodular triangulation, then $rP$ has a unimodular triangulation;
\item if $P$ has a regular unimodular triangulation, then $rP$ has a regular unimodular triangulation;
\item if $P$ has a quadratic triangulation, then $rP$ has a quadratic triangulation.
\end{enumerate}
\end{theorem}

\section{Concluding remarks}\label{Section-Finily}

The main contribution of this paper is to complete all the characterizations in \cref{tab:coefficients}. 
In particular, it enables the reader to clearly understand the results on the preservation properties of the eight common operations on lattice polytopes.
This may be crucial for future research on Ehrhart theory of polytopes.
For example, it provides valuable guidance for constructing counterexamples to some conjectures.

We again mention that the direct motivation of this paper is that Ferroni \cite{Ferroni2026} used the Cayley sum to construct counterexamples: lattice polytopes with the integer decomposition property whose $h^*$-polynomial is not unimodal.
This elegant construction immediately attracted the author's attention.
This prompted the authors to complete \cref{tab:coefficients}.

We also mention the following two conjectures. These are problems of recent interest to the author.
Although the author's attempts to find answers were unsuccessful, the results of this paper may offer some hope for these problems.

\begin{question}(Oda's smoothness problem; \cite{Oda})
Does every smooth lattice polytope have the integer decomposition property?
\end{question}

Here smoothness has the edge basis meaning defined in \cref{sec:cayley}: at every vertex of a $d$-polytope, exactly $d$ edges meet and their primitive directions form a basis of $\Z^d$. The question originates in \cite{Oda} and remains open in the account \cite{SWZ}. It asks whether, for every smooth $P$ and every integer $n\geq1$, one has
\[(P\cap\Z^d)+(nP\cap\Z^d)=((n+1)P)\cap\Z^d.
\]

Indeed, IDP gives this equality by taking one summand separately in a decomposition into $n+1$ lattice points.
Conversely, if the equality holds for every $n$, induction on $n$ decomposes every point of $nP$ into $n$ points
of $P$. Thus the question is directly related to lattice point addition in Minkowski sums. The non-IDP Minkowski and Cayley constructions above do not resolve it, since smoothness of their outputs has not been established.

\begin{question}(Gal's conjecture for $h^*$-polynomials; \cite{SRGal})
If a Gorenstein lattice polytope admits a quadratic triangulation, is its $h^*$-polynomial $\gamma$-positive?
\end{question}

This problem can be traced back to the work of Gal \cite{SRGal} and also appears in \cite[Conjecture~3.28]{FH}.
The flag condition in a quadratic triangulation is essential to the formulation. Without it, even reflexive polytopes
with regular unimodular triangulations can have negative $\gamma$-coefficients, see Ferroni's example \cite[Corollary~2.2]{FerroniGamma}.






\vspace{5mm}
\noindent
{\small \textbf{Acknowledgments:}}
The author would like to thank his colleague, Dr. Chao Xu, for providing the inspiration for this paper, which arose from Dr. Xu's inquiry as to whether there are any nice operations between lattice polytopes.
The author would also like to thank Dr. Zihao Zhang for many helpful suggestions.
In particular, Dr. Zhang provided considerable assistance with the proofs of the theorems in the appendix.
The author was partially supported by the Postdoctoral Fellowship Program and China Postdoctoral Science Foundation (Grant No. BX2026002).

\noindent{\small \textbf{Declaration of AI Assistance:}}
During the preparation of this manuscript, the author used ChatGPT to assist in exploring possible approaches, checking technical details, and improving the exposition. All mathematical arguments, computations, proofs, and results were independently verified by the author. The author take full responsibility for the content of this manuscript.

\newpage
\appendix

\section{Proof of \cref{thm:dilation-ulc-general}}\label{Append-Section}

\subsection{Lorentzian polynomials}

Let $A$ be a nonnegative integer and let $0\leq\alpha\leq1$.
For a nonzero polynomial $f(z)=\sum_{i=0}^A a_i z^i$ with nonnegative coefficients, write $f(z)\in\ULC_\alpha(A)$ if the set
$\{i:a_i>0\}$ is an interval of integers and
\[\left(\frac{a_i}{\binom Ai^\alpha}\right)^2\geq\frac{a_{i-1}}{\binom A{i-1}^\alpha}\frac{a_{i+1}}{\binom A{i+1}^\alpha},
 \qquad 1\leq i\leq A-1.
\]
An interval support means that there are no zero coefficients between two positive coefficients. Initial and terminal zeros are permitted; in particular, $A$ may exceed $\deg f(z)$. The inequalities are trivially satisfied for $A \le 1$.
Thus $\ULC_0(A)$ consists of the log-concave polynomials with no internal zeros, and $\ULC_1(A)=\ULC(A)$.
For the auxiliary polynomials in this and the following two subsections, the notation $\ULC(A)$ allows initial zeros
subject to the same interval support condition.

For $0\leq\alpha\leq\beta\leq1$, these classes satisfy $\ULC_\beta(A)\subseteq\ULC_\alpha(A)$. To see this, write
\[\frac{a_i}{\binom Ai^\alpha}=\frac{a_i}{\binom Ai^\beta}\binom Ai^{\beta-\alpha}.
\]
The binomial sequence is positive and log-concave, since its successive ratios $(A-i)/(i+1)$ decrease with $i$. 
Since $\beta - \alpha \ge 0$, raising the entries to this power preserves log-concavity. The requisite inequalities then follow immediately from the term-wise product of the log-concavity conditions of both factors. Furthermore, as the binomial factor is positive, the support remains invariant.

We shall use the Lorentzian results of Br\"and\'en and Huh \cite{BH} to prove that convolution preserves every one of these
classes. We first specify the definitions and the forms of their results that are needed.

For a vector $\nu=(\nu_1,\ldots,\nu_m)$ of nonnegative integers and a list of variables $w=(w_1,\ldots,w_m)$, set
\[ |\nu|=\sum_{j=1}^m\nu_j,\qquad \nu!=\prod_{j=1}^m\nu_j!,\qquad w^\nu=\prod_{j=1}^m w_j^{\nu_j},\qquad
 \partial^\nu=\prod_{j=1}^m\left(\frac{\partial}{\partial w_j}\right)^{\nu_j}.
\]
The \emph{support} of a polynomial $H$ (denoted as $\operatorname{supp}H$) is the set of exponent vectors of its nonzero monomials, and $[w^\nu]H$ denotes the coefficient of $w^\nu$. A polynomial is homogeneous of degree $D$ if all its nonzero monomials have total degree $D$. 
A finite set $J$ of nonnegative integer vectors of equal total size is \emph{$M$-convex} if it satisfies the following exchange condition: whenever $\nu,\mu\in J$ and $\nu_i>\mu_i$, there is an index $j$ with $\nu_j<\mu_j$ such that both $\nu-e_i+e_j$ and $\mu+e_i-e_j$ belong to $J$, where $e_i$ is the $i$-th coordinate vector.

For $D\geq2$, a homogeneous polynomial $H$ of degree $D$ with nonnegative coefficients is \emph{Lorentzian} if and only if its support is $M$-convex and, for every $\eta$ with $|\eta|=D-2$, the Hessian of $\partial^\eta H$ has at most one positive eigenvalue
\cite[Theorem~2.25]{BH}. Here the Hessian is the symmetric matrix of second partial derivatives; it is constant because
$\partial^\eta H$ is quadratic. Nonnegative homogeneous polynomials of degree zero or one, as well as the zero polynomial,
are included in the Lorentzian class. In two variables, the characterization takes a particularly simple form:
\[\sum_{i=0}^D d_i w_1^i w_2^{D-i}\text{ is Lorentzian}\quad\Longleftrightarrow\quad
\left(\frac{d_i}{\binom Di}\right)_{i=0}^D \text{ is log-concave with no internal zeros},
\]
provided the coefficients are nonnegative and the polynomial is nonzero \cite[Example~2.26]{BH}.

We know that the products of Lorentzian polynomials are Lorentzian \cite[Corollary~2.32]{BH}. 
We shall also use the following operation on their normalized coefficients. Set 
\[ H(w)=\sum_{\nu\in\operatorname{supp}H}\frac{c_\nu}{\nu!}w^\nu,\qquad c_\nu=\nu![w^\nu]H.
\]
For $0\leq p\leq1$, define
\[ R_p(H)=\sum_{\nu\in\operatorname{supp}H}\frac{c_\nu^p}{\nu!}w^\nu.
\]
If $H$ is Lorentzian, then $R_p(H)$ is Lorentzian \cite[Proposition~3.25]{BH}. 
At $p=0$, the sum is still taken over the original support. Thus no convention for $0^0$ is required.

The final result we need concerns linear operators. Let $\xi=(\xi_1,\ldots,\xi_p)$ be the input variables and let
$w=(w_1,\ldots,w_m)$ be the output variables; the two lists may have different lengths. For a nonnegative integer vector
$\kappa=(\kappa_1,\ldots,\kappa_p)$, let $\R_\kappa[\xi]$ denote the space spanned by all monomials
$\xi^\nu$ with $0\leq\nu_j\leq\kappa_j$. Define $\R_\gamma[w]$ in the same way for an output bound $\gamma$.
Suppose that
\[ T:\R_\kappa[\xi]\longrightarrow\R_\gamma[w]
\]
is linear and homogeneous of degree $\ell$: for every monomial in the full input space, $T(\xi^\nu)$ is either zero or homogeneous
of degree $|\nu|+\ell$. With independent auxiliary variables $u=(u_1,\ldots,u_p)$, its symbol is
\[\operatorname{sym}T(w,u)=\sum_{0\leq\nu\leq\kappa}\binom\kappa\nu T(\xi^\nu)u^{\kappa-\nu},
 \qquad
 \binom\kappa\nu=\prod_{j=1}^p\binom{\kappa_j}{\nu_j}.
\]
The vector inequalities are coordinatewise. This symbol is homogeneous of degree $|\kappa|+\ell$ in $m+p$ variables.
If the symbol is Lorentzian, then $T$ sends every homogeneous Lorentzian polynomial in $\R_\kappa[\xi]$ to a Lorentzian
polynomial \cite[Theorem~3.2]{BH}. The theorem applies to the entire space $\R_\kappa[\xi]$, rather than only to monomials of
one fixed total degree.

Adjoining variables that do not occur in a polynomial does not affect the Lorentzian property. Indeed, this extends its support by zero coordinates and preserves the exchange condition. Derivatives involving an added variable vanish; the Hessians of all other
derivatives are extended by zero rows and columns. This observation allows us to apply the product theorem to polynomials initially
written in separate lists of variables.

\begin{lemma}\label{lem:fractional-ulc-convolution}
Let $0\leq\alpha\leq1$. If $f\in\ULC_\alpha(A)$ and $g\in\ULC_\alpha(B)$, then $fg\in\ULC_\alpha(A+B)$.
\end{lemma}
\begin{proof}
Write $f(z)=\sum_{i=0}^A a_i z^i$ and $g(z)=\sum_{i=0}^B b_i z^i$, and homogenize them as
\[ F(w_1,w_2)=\sum_{i=0}^A a_i w_1^i w_2^{A-i},
 \qquad
 G(w_1,w_2)=\sum_{i=0}^B b_i w_1^i w_2^{B-i}.
\]
On a polynomial ring in any finite list of variables, define a linear operator by the monomial rule
\[ N_\alpha(w^\nu)=\frac{w^\nu}{(\nu!)^{1-\alpha}}.
\]
For $N_\alpha(F)$, the sequence in the two variable Lorentzian characterization is
\[\frac{a_i}{(i!(A-i)!)^{1-\alpha}\binom Ai}=\frac{a_i(i!(A-i)!)^\alpha}{A!}=\frac{a_i}{\binom Ai^\alpha}(A!)^{\alpha-1}.
\]
The last factor is a positive constant independent of $i$. Thus $N_\alpha(F)$ is Lorentzian by the hypothesis on $f$.
The same calculation shows that $N_\alpha(G)$ is Lorentzian. It remains to show that $N_\alpha(FG)$ is Lorentzian.

Set $K=\max\{A,B\}$. 
Let $x=(x_1,x_2)$ and $y=(y_1,y_2)$ be two independent pairs of variables. 
We define a linear operator on the full coordinate-degree space with bounds $(K,K,K,K)$ by its action on monomials:
\[
  C_\alpha(x^\nu y^\mu)=\left(\frac{\nu!\mu!}{(\nu+\mu)!}\right)^{1-\alpha} w^{\nu+\mu},
  \qquad
  0\leq\nu_j,\mu_j\leq K\quad(j=1,2).
\]
Its output bounds are $(2K,2K)$. It is homogeneous of degree zero on the entire input space. Its symbol (with auxiliary lists
$u=(u_1,u_2)$ and $v=(v_1,v_2)$) is
\[\prod_{j=1}^2\left(\sum_{a,b=0}^K\binom Ka\binom Kb\left(\frac{a!b!}{(a+b)!}\right)^{1-\alpha}w_j^{a+b}u_j^{K-a}v_j^{K-b}\right).
\]
This constitutes a homogeneous polynomial of degree $4K$ in six variables, consistent with the requirements for an operator having four input variables, two output variables, and a zero degree shift.

We now consider one factor, with variables denoted by $w,u,v$. The exponent triple $(a+b,K-a,K-b)$ determines $(a,b)$ uniquely.
Consequently, by multiplying the coefficient of this monomial by its monomial factorial, we obtain its normalized coefficient without merging any terms.
That coefficient is
\[ \binom Ka\binom Kb \left(\frac{a!b!}{(a+b)!}\right)^{1-\alpha} (a+b)!(K-a)!(K-b)!
=(K!)^2\left(\frac{(a+b)!}{a!b!}\right)^\alpha =(K!)^2\binom{a+b}{a}^{\alpha}.
\]
On the other hand, the normalized coefficient of the same monomial in $(w+u)^K(w+v)^K$ is
$ (K!)^2\binom{a+b}{a}$. It follows, coefficient by coefficient, that the factor in the symbol equals
\[ (K!)^{2(1-\alpha)} R_\alpha\bigl((w+u)^K(w+v)^K\bigr).
\]
The linear forms $w+u$ and $w+v$ are Lorentzian. Their powers and their product are Lorentzian by the product theorem, including
$K=0$, when the product is the constant polynomial $1$.
The normalized coefficient power theorem shows that the displayed factor is Lorentzian. Multiplication by its positive constant
preserves its support and multiplies every relevant Hessian by that constant. Therefore, it preserves the Lorentzian property.
The product theorem then shows that the full symbol of $C_\alpha$ is Lorentzian. 
All the hypotheses of the operator theorem have now been verified. 
Thus $C_\alpha$ preserves the Lorentzian property on its input space.

The polynomials $N_\alpha(F)(x)$ and $N_\alpha(G)(y)$ are Lorentzian in the four input variables, and their product is Lorentzian. 
This product belongs to the input space because each coordinate exponent is at most $K$. 
For every pair of monomials, the factorial factors cancel as follows:
\[ C_\alpha\left( \frac{x^\nu y^\mu}{(\nu!\mu!)^{1-\alpha}} \right)
 =\frac{w^{\nu+\mu}}{((\nu+\mu)!)^{1-\alpha}}.
\]
By applying this identity to all monomial pairs and using linearity, we have
\[ C_\alpha\bigl(N_\alpha(F)(x)N_\alpha(G)(y)\bigr) =N_\alpha(FG)(w).
\]
Terms with the same exponent $\nu+\mu$ are summed on both sides, exactly as in the product $FG$. 
Therefore $N_\alpha(FG)$ is Lorentzian by the operator theorem.

Finally, $FG$ is the homogenization of $fg$ in order $A+B$. 
For $0\leq i\leq A+B$, the same normalization calculation gives
\[\frac{[w_1^iw_2^{A+B-i}]N_\alpha(FG)}{\binom{A+B}i}=\frac{[z^i](fg)}{\binom{A+B}i^\alpha}((A+B)!)^{\alpha-1}.
\]
The sequences on the two sides differ by one positive constant, so they have the same support and satisfy the same log-concavity
inequalities. The Lorentzian property of $N_\alpha(FG)$ therefore gives exactly the defining condition for $\ULC_\alpha(A+B)$.
The product $fg$ is nonzero because both factors are nonzero with nonnegative coefficients. Hence $fg\in\ULC_\alpha(A+B)$.
\end{proof}

\subsection{Basic lemmas}

Recalling the definition of $h^{\langle r,0\rangle}(z)$ from \cref{Subsection-veronese}, we state the general case of the definition.
For a positive integer $r$ and an integer $e$ with $0\leq e<r$, define the \emph{residue section} of $h(z)=\sum_{n\geq0}a_nz^n$ by
\[ h^{\langle r,e\rangle}(z)=\sum_{k\geq0}a_{rk+e}z^k.
\]
Thus the coefficient of $z^k$ in this section is the coefficient of $z^{rk+e}$ in $h(z)$. The case $e=0$ is the extraction used in
\eqref{eq:dilation-extraction}. We first prove that a suitable fractional binomial normalization becomes ordinary ultra log-concavity after extraction. We then establish that normalization for the third and fourth powers of $B_r(z)=1+z+\cdots+z^{r-1}$.

\begin{lemma}\label{lem:logarithmic-comparison}
The function
\[f(x)=\log\left(1+\frac1x\right),\qquad x>0,
\]
is positive and decreasing, and $\log f$ is convex. For every fixed $y\geq0$, the quotient $f(x+y)/f(x)$ is nondecreasing in $x$.
\end{lemma}
\begin{proof}
The positivity follows from $1+\frac{1}{x}>1$. We have
\[ f'(x)=-\frac1{x(x+1)},\qquad f''(x)=\frac{2x+1}{x^2(x+1)^2}.
\]
In particular, $f(x)$ is decreasing. The second derivative of its logarithm is
\[(\log f)''(x)=\frac{f(x)f''(x)-f'(x)^2}{f(x)^2}=\frac{(2x+1)f(x)-1}{x^2(x+1)^2f(x)^2}.
\]
The integral representation
\[
 f(x)=\int_x^{x+1}\frac{dt}{t}\geq\frac1{x+1}
\]
implies $(2x+1)f(x)-1\geq \frac{x}{x+1}>0$. Hence $\log f(x)$ is convex, and its first derivative is nondecreasing. 
For fixed $y\geq0$, it follows that
\[\frac{d}{dx}\log\frac{f(x+y)}{f(x)}= (\log f)'(x+y)-(\log f)'(x)\geq0.
\]
Since the quotient is positive, this proves the assertion.
\end{proof}

\begin{lemma}\label{lem:fractional-ulc-section}
Let $r\geq2$, $N\geq0$, and $0\leq e<r$ be integers, with $e\leq N$, and set $D=\lfloor(N-e)/r\rfloor$. For $1\leq k\leq D-1$, we have
\[\left(\frac{\binom{N}{rk+e}^2}{\binom{N}{r(k-1)+e}\binom{N}{r(k+1)+e}}\right)^{53/(40r)}
 \geq\frac{\binom{D}{k}^2}{\binom{D}{k-1}\binom{D}{k+1}}.
\]
Consequently, if $h(z)\in\ULC_{53/(40r)}(N)$, then $h^{\langle r,e\rangle}(z)$ is zero or belongs to $\ULC(D)$.
\end{lemma}
\begin{proof}
If $D\leq1$, the case is trivial. Suppose $D\geq2$ and fix $1\leq k\leq D-1$. Write $N-e=rD+b$, where $0\leq b<r$, and set $n=rk+e$. Then $n-r\geq0$ and $n+r\leq N$. By canceling the factorials, we obtain
\begin{align*}
 \frac{\binom{N}{n}^2}{\binom{N}{n-r}\binom{N}{n+r}}
 &=\frac{(n-r)!(n+r)!}{(n!)^2}\frac{(N-n+r)!(N-n-r)!}{((N-n)!)^2}\\
 &=\prod_{j=1}^{r}\left(1+\frac{r}{n-r+j}\right)\left(1+\frac{r}{N-n-r+j}\right).
\end{align*}
Because $e,b\leq r-1$, the positive denominators satisfy
\begin{align*}
 n-r+j&=r(k-1)+e+j\leq rk+j-1,\\
 N-n-r+j&=r(D-k-1)+b+j\leq r(D-k)+j-1.
\end{align*}
It follows that
\begin{align*}
 \log\left(1+\frac{r}{n-r+j}\right)&\geq f\left(k+\frac{j-1}{r}\right),\\
 \log\left(1+\frac{r}{N-n-r+j}\right)&\geq f\left(D-k+\frac{j-1}{r}\right),
\end{align*}
where $f$ is the function in \cref{lem:logarithmic-comparison}.

For $x>0$, set $I(x)=\int_0^1f(x+t)\,dt$. On each of the $r$ subintervals of length $1/r$, the decreasing function $t\mapsto f(x+t)$ does not exceed its value at the left endpoint. 
By integrating over these subintervals and summing the results, we have
\[\frac1r\sum_{j=1}^{r}f\left(x+\frac{j-1}{r}\right)\geq I(x).
\]
For each $t\in[0,1]$, \cref{lem:logarithmic-comparison} states that $f(x+t)/f(x)$ is nondecreasing in $x$. Thus, for $x\geq1$,
\[\frac{I(x)}{f(x)}=\int_0^1\frac{f(x+t)}{f(x)}\,dt\geq\int_0^1\frac{f(1+t)}{f(1)}\,dt
=\frac{3\log3-4\log2}{\log2}.
\]
For the final equality, $f(1)=\log2$, and integration of $\log(u+1)-\log u$ from $u=1$ to $u=2$ gives $I(1)=3\log3-4\log2$. 
The choice of the exponent is justified by
\[\frac{53}{40}(3\log3-4\log2)\geq\log2.
\]
Indeed, this is equivalent to $3^{53}\geq2^{84}$.
Consequently, we obtain $(53/40)I(x)\geq f(x)$ for every $x\geq1$.
By applying the preceding estimates at $x=k$ and $x=D-k$, we prove
\begin{align*}
\frac{53}{40r}\log\frac{\binom{N}{n}^2}{\binom{N}{n-r}\binom{N}{n+r}}
 &\geq\frac{53}{40}\bigl(I(k)+I(D-k)\bigr)\geq f(k)+f(D-k)\\
 &=\log\frac{(k+1)(D-k+1)}{k(D-k)}=\log\frac{\binom{D}{k}^2}{\binom{D}{k-1}\binom{D}{k+1}}.
\end{align*}
By taking the exponential of both sides, we obtain the required binomial inequality.

Set $h(z)=\sum_{n=0}^{N}a_nz^n$ and $q_n=a_n/\binom{N}{n}^{53/(40r)}$. 
The hypothesis states that $(q_n)$ is log-concave and its positive entries form an interval. 
We claim that $q_n^2\geq q_{n-r}q_{n+r}$. 
If either term on the right is zero, the claim follows from nonnegativity. 
Otherwise every entry from $q_{n-r}$ to $q_{n+r}$ is positive. 
Ordinary log-concavity says that the ratios $q_j/q_{j-1}$ are nonincreasing on this interval. Therefore, 
\[\frac{q_n}{q_{n-r}}=\prod_{j=1}^{r}\frac{q_{n-r+j}}{q_{n-r+j-1}}\geq\prod_{j=1}^{r}\frac{q_{n+j}}{q_{n+j-1}}
 =\frac{q_{n+r}}{q_n}.\]
This proves the claim. By combining it with the binomial inequality, we have
\[a_{rk+e}^2\geq\frac{\binom{D}{k}^2}{\binom{D}{k-1}\binom{D}{k+1}}a_{r(k-1)+e}a_{r(k+1)+e},\qquad 1\leq k\leq D-1.
\]
Therefore, the sequence of coefficients of the section (when divided by $\binom{D}{k}$) is log-concave. 

Furthermore, the support of $(a_n)$ forms an interval. 
Thus, any intermediate coefficient between two positive entries of the section belongs to this support, forcing the corresponding coefficient to be positive.
Thus a nonzero section has no internal zeros and belongs to $\ULC(D)$.
\end{proof}

We next estimate the coefficients of the uniform powers. For a polynomial $F$, write $[z^n]F$ for the coefficient of $z^n$. If $R,L\geq1$ are integers, then
\[[z^n](1+z+\cdots+z^R)^L=\#\{(m_1,\ldots,m_L)\in\Z^L: 0\leq m_i\leq R,\ m_1+\cdots+m_L=n\}.
\]
An ordered tuple of nonnegative integers with sum $n$ is a \emph{weak composition} of $n$. 
Without the upper bounds on the entries, the number of weak compositions into $L$ parts is $\binom{n+L-1}{L-1}$, see \cite[Chapter~1]{RP.Stanley}.
For any $0 \leq n \leq LR$, one obtains an admissible tuple by assigning $\min(R,n)$ to the first entry and proceeding inductively with the residual sum.
At each step the remaining total is at most $R$ times the number of remaining entries. Thus all coefficients in this range are positive. The bijection $(m_1,\ldots,m_L)\mapsto(R-m_1,\ldots,R-m_L)$ also shows that the
coefficient sequence is symmetric about $LR/2$.

For a positive sequence $(c_n)_{n=0}^{N}$, the binomial normalization sequence $\left( c_n / \binom{N}{n}^\alpha \right)_{n=0}^N$ is log-concave if and only if
\[\frac{c_n^2}{c_{n-1}c_{n+1}}\geq\left(\frac{\binom{N}{n}^2}{\binom{N}{n-1}\binom{N}{n+1}}\right)^\alpha=\left(1+\frac{N+1}{n(N-n)}\right)^\alpha
\]
for all $1\leq n\leq N-1$. The subsequent two lemmas establish these inequalities for the cases $N=3R$ and $N=4R$, utilizing the uniform exponent $\alpha=53/(40(R+1))$.

\begin{lemma}\label{lem:uniform-third-fractional}
Let $R\geq2$ be an integer and set
\[ c_n=[z^n](1+z+\cdots+z^R)^3,\qquad 0\leq n\leq3R.
\]
Then the sequence 
\[ \left(\frac{c_n}{\binom{3R}{n}^{\,53/(40(R+1))}}\right)_{n=0}^{3R}
\]
is log-concave.
\end{lemma}
\begin{proof}
Set $\alpha=53/(40(R+1))$, so $0<\alpha<1$. Both $(c_n)$ and $(\binom{3R}{n})$ are positive and symmetric. It suffices to prove the
normalized log-concavity inequality for $1\leq n\leq\lfloor3R/2\rfloor$.

Suppose first that $1\leq n\leq R-1$. For $0\leq m\leq R$, every weak composition of $m$ into three parts has all entries at most $R$, so $c_m=\binom{m+2}{2}$. By applying this formula at $m=n-1,n,n+1$, we have
\[\frac{c_n^2}{c_{n-1}c_{n+1}}=\frac{(n+1)(n+2)}{n(n+3)}.
\]
According to
$$\log \frac{(n+1)(n+2)}{n(n+3)} = \log \frac{n+1}{n} - \log \frac{n+3}{n+2}= \log\left(1 + \frac{1}{n}\right) - \log\left(1 + \frac{1}{n+2}\right)$$
and $f(x) = \log(1+1/x)$, the logarithm of this ratio simplifies to $f(n) - f(n+2)$.
For the binomial sequence $\binom{3R}{n}$, we have
$$\log\frac{\binom{3R}{n}^2}{\binom{3R}{n-1}\binom{3R}{n+1}}
=\log\frac{n+1}{n} \cdot \frac{3R-n+1}{3R-n} 
= \log\left(1 + \frac{1}{n}\right)+\log \left(1 + \frac{1}{3R-n}\right)=f(n) + f(3R-n).$$
Their quotient is
\[
 \frac{f(n)-f(n+2)}{f(n)+f(3R-n)}
 =\frac{1-f(n+2)/f(n)}{1+f(3R-n)/f(n)}.
\]
\cref{lem:logarithmic-comparison} implies that $f(n+2)/f(n)$
is nondecreasing in $n$. Furthermore, $f(3R-n)$ is nondecreasing and
$f(n)$ is positive and nonincreasing, so $f(3R-n)/f(n)$ is nondecreasing.
The numerator of the last quotient is positive and nonincreasing,
and its denominator is positive and nondecreasing. The quotient is
therefore minimized at $n=R-1$. 
Therefore, we have
$$\frac{f(n)-f(n+2)}{f(n)+f(3R-n)}\geq \frac{f(R-1) - f(R+1)}{f(R-1) + f(2R+1)}= \frac{\log\frac{R(R+1)}{(R-1)(R+2)}}{\log\frac{2R(R+1)}{(R-1)(2R+1)}}.$$
It remains to prove
\[40(R+1)\log\frac{R(R+1)}{(R-1)(R+2)}\geq53\log\frac{2R(R+1)}{(R-1)(2R+1)}.
\]

We prove this inequality for real $R\geq2$. Define the difference
\[\Phi(R)=40(R+1)\log\frac{R(R+1)}{(R-1)(R+2)}-53\log\frac{2R(R+1)}{(R-1)(2R+1)}.
\]
For $x\geq1$, the inequality $\log x\leq(x-x^{-1})/2$ follows by subtracting the left side from the right side: the difference vanishes at $x=1$, and its derivative is $(x-1)^2/(2x^2)\geq0$. Hence
\[ \log\frac{R(R+1)}{(R-1)(R+2)} \leq\frac{2(R^2+R-1)}{R(R-1)(R+1)(R+2)}.
\]
Furthermore, we have
\begin{align*}
\Phi'(R)&=40\log\frac{R(R+1)}{(R-1)(R+2)}+40(R+1)\left(\frac1R+\frac1{R+1}-\frac1{R-1}-\frac1{R+2}\right)
\\&\qquad\qquad -53\left(\frac1R+\frac1{R+1}-\frac1{R-1}-\frac2{2R+1}\right).
\end{align*}
By replacing the logarithm by its upper bound and writing the rational terms over the common denominator, we obtain
\[\Phi'(R)\leq-\frac{R^3-24R^2+215R+54}{R(R-1)(R+1)(R+2)(2R+1)}<0.
\]
The denominator is positive. The numerator is positive because its value at $R=2$ is $396$, and its derivative is
$3(R-8)^2+23>0$. 
Finally, we prove that $\Phi(R)\to0$ as $R\to\infty$. For the first term, the estimate $\log x\leq x-1$ gives
\[0\leq(R+1)\log\frac{R(R+1)}{(R-1)(R+2)}\leq\frac{2(R+1)}{(R-1)(R+2)}\longrightarrow0;
\]
the second logarithm also tends to 0. 
Since $\Phi$ is decreasing with limit zero, we obtain $\Phi(R)\geq0$ for $R\geq2$. 

We have shown that the value of the quotient at the minimum point $n=R-1$ is at least $\alpha=\frac{53}{40(R+1)}$.
Therefore, we obtain 
\[ f(n)-f(n+2)\geq\alpha \bigl(f(n)+f(3R-n)\bigr),\qquad 1\leq n\leq R-1.
\]
By taking the exponential, we obtain:
$$\frac{c_n^2}{c_{n-1}c_{n+1}}=e^{f(n)-f(n+2)}\geq e^{\alpha(f(n)+f(3R-n))}=\left(\frac{\binom{3R}{n}^2}{\binom{3R}{n-1}\binom{3R}{n+1}}\right)^\alpha.$$
This proves the normalized log-concavity inequality on the first range $1\leq n\leq R-1$.

Now suppose $R\leq n\leq\lfloor3R/2\rfloor$. 
For each $m\in\{n-1,n,n+1\}$, an application of the inclusion-exclusion principle to the weak compositions of $m$ into three parts yields
\[c_m=\frac{(m+1)(m+2)}{2}-3\frac{(m-R)(m-R+1)}{2}.
\]
The subtracted term accounts for configurations where $m\geq R+1$ and a specified entry exceeds $R$. 
By subtracting $R+1$ from the designated part, we obtain a bijection with the weak compositions of $m-R-1$ into three parts. This yields exactly $\binom{m-R+1}{2}$ valid configurations for each of the three possible choices of the part.
Because $m \leq \lfloor 3R/2 \rfloor + 1 < 2(R+1)$, no two parts can simultaneously exceed $R$, which guarantees that no configuration is subtracted twice. 
Furthermore, the subtracted polynomial evaluates to zero for $m \in \{R-1, R\}$, ensuring the formula's validity across the entire range.
An expansion of this identity ultimately produces
\[c_m=\frac{3}{4}R^2+\frac{3}{2}R+1-\left(m-\frac{3}{2}R\right)^2.
\]
Set
\[ H=\frac34R^2+\frac32R+1,\qquad w=\left(n-\frac32R\right)^2.
\]
Then $0\leq w\leq \frac{R^2}{4}$. By using the formula for the three coefficients, we obtain
\begin{align*}
 c_n&=H-w,\qquad c_{n-1}c_{n+1}=(H-1-w)^2-4w,\\
 c_n^2-c_{n-1}c_{n+1}&=2H-1+2w,\qquad n(3R-n)=\frac94R^2-w.
\end{align*}
All three coefficients are positive. Thus the denominator below is positive. The identities yield
\[
 n(3R-n)\left(\frac{c_n^2}{c_{n-1}c_{n+1}}-1\right)
 =\frac{(9R^2/4-w)(2H-1+2w)}{(H-1-w)^2-4w}.
\]
This expression is at least its value at $w=0$. 
Indeed, after multiplying by the positive denominators, the difference is
\begin{align*}
 &\left(\frac{9R^2}{4}-w\right)(2H-1+2w)(H-1)^2-\frac{9R^2}{4}(2H-1)\bigl((H-1-w)^2-4w\bigr)
\\=& \frac{9R^2w(R+1)}{16}\bigl(12R^3+33R^2+32R+12-(8R+12)w\bigr).
\end{align*}
The identity follows by expanding in $w$ and substituting the displayed
value of $H$. Its right side is nonnegative, since $w\leq \frac{R^2}{4}$ implies
\[12R^3+33R^2+32R+12-(8R+12)w>0.\]
Consequently, we obtain 
\begin{align*}
n(3R-n)\left(\frac{c_n^2}{c_{n-1}c_{n+1}}-1\right)\geq\frac{(9R^2/4)(2H-1)}{(H-1)^2}=\frac{2(3R^2+6R+2)}{(R+2)^2}.
\end{align*}
The lower bound exceeds $\alpha(3R+1)$, as the following identity shows
\begin{align*}
\frac{2(3R^2+6R+2)}{(R+2)^2}-\frac{53(3R+1)}{40(R+1)}=\frac{81(R-2)^3+517(R-2)^2+888(R-2)+304}{40(R+1)(R+2)^2}>0.
\end{align*}
We have therefore proved
\[\frac{c_n^2}{c_{n-1}c_{n+1}} \geq1+\alpha\frac{3R+1}{n(3R-n)}.
\]
For $0<\alpha<1$, concavity of $x\mapsto(1+x)^\alpha$ on $[0,\infty)$ bounds this function by its tangent at zero:
$(1+x)^\alpha\leq1+\alpha x$. By applying this inequality, we obtain
\[\frac{c_n^2}{c_{n-1}c_{n+1}}\geq\left(1+\frac{3R+1}{n(3R-n)}\right)^\alpha=\left(\frac{\binom{3R}{n}^2}
{\binom{3R}{n-1}\binom{3R}{n+1}}\right)^\alpha.
\]
This proves normalized log-concavity on the second range. The first range was proved above, and symmetry covers the remaining indices. This completes the proof.
\end{proof}

\begin{lemma}\label{lem:uniform-fourth-fractional}
Let $R\geq2$ be an integer and set
\[c_n=[z^n](1+z+\cdots+z^R)^4,\qquad 0\leq n\leq4R.
\]
Then we have 
\[\frac{c_n^2}{c_{n-1}c_{n+1}}\geq1+\frac{6R}{(R+1)n(4R-n)},\qquad 1\leq n\leq4R-1.
\]
In particular, the sequence 
$$\left(\frac{c_n}{\binom{4R}{n}^{53/(40(R+1))}}\right)_{n=0}^{4R}$$ is log-concave.
\end{lemma}
\begin{proof}
We know that all coefficients are positive, and $c_n=c_{4R-n}$. It suffices to prove the rational inequality for $1\leq n\leq2R$.
For $1\leq n\leq R-1$, the upper bound on the entries excludes no weak composition at the three indices $n-1,n,n+1$. Thus
$c_m=\binom{m+3}{3}$ at these indices. Therefore, we have
\[ \frac{c_n^2}{c_{n-1}c_{n+1}} =\frac{(n+1)(n+3)}{n(n+4)} =1+\frac3{n(n+4)}.
\]
It remains to compare $3/(n(n+4))$ with $6R/((R+1)n(4R-n))$. Since the denominators are positive, the required comparison is equivalent to
\[3(R+1)(4R-n)-6R(n+4)\geq0.
\]
The left side is a decreasing function of $n$. Its value at $n=R-1$ is
\[3(R+1)(3R+1)-6R(R+3)=3(R-1)^2\geq0.
\]
This proves the assertion on the first range.

For $R \leq n \leq 2R$ and $m \in \{n-1, n, n+1\}$, we have
\[
  6c_m = (m+1)(m+2)(m+3) - 4(m-R)(m-R+1)(m-R+2).
\]
By inclusion-exclusion, for $m \geq R+1$, any weak composition with a part exceeding $R$ reduces to a weak composition of $m-R-1$ into four parts by subtracting $R+1$ from the violating part. 
This accounts for $4\binom{m-R+2}{3}$ configurations, corresponding to the four choices for the violating part. 
The bound $m \leq 2R+1 < 2(R+1)$ guarantees that at most one part can exceed $R$. 
Finally, since the subtracted polynomial evaluates to zero for $m \in \{R-1, R\}$, 
the formula holds universally across the specified range.

To display the sign of the resulting polynomial inequality, we set $a=2R-n$ and $b=n-R$. Then we have $a,b\geq0$, $R=a+b$, and $n=a+2b$. The three formulas are
\begin{align*}
 6c_{n-1}&=(a+2b)(a+2b+1)(a+2b+2)-4(b-1)b(b+1),\\
 6c_n&=(a+2b+1)(a+2b+2)(a+2b+3)-4b(b+1)(b+2),\\
 6c_{n+1}&=(a+2b+2)(a+2b+3)(a+2b+4)-4(b+1)(b+2)(b+3).
\end{align*}
A direct substitution of these relations, followed by an expansion in powers of $b$, yields
\begin{align*}
12\bigl((R+1)n(4R-n)c_n^2-((R+1)n(4R-n)+6R)c_{n-1}c_{n+1}\bigr)=\sum_{j=0}^{7}A_j(a)b^j,
\end{align*}
where the complete list of coefficient polynomials is
\begin{align*}
A_0(a)&=a^2(a+1)(a+2)(a^3+24a^2+5a+18)
\\ A_1(a)&=a(33a^5+256a^4+395a^3+248a^2+100a+96)
\\ A_2(a)&=4(47a^5+183a^4+145a^3+50a^2+48a+12)
\\ A_3(a)&=4(105a^4+212a^3+112a^2+64a+36)
\\ A_4(a)&=8(55a^3+74a^2+52a+26)
\\ A_5(a)&=16(17a^2+24a+13)
\\ A_6(a)&=128(a+1)
\\ A_7(a)&=32
\end{align*}
Here every coefficient is nonnegative for $a\geq0$. As $a+b=R\geq2$, either $a>0$, in which
case $A_0(a)>0$, or $b>0$, in which case $A_7(a)b^7>0$. Thus the expression inside parentheses on the left is positive. After division by $(R+1)n(4R-n)c_{n-1}c_{n+1}>0$, this gives
\[\frac{c_n^2}{c_{n-1}c_{n+1}}\geq1+\frac{6R}{(R+1)n(4R-n)}.
\]
By symmetry, we prove the rational inequality at every remaining index.

Set $\alpha=53/(40(R+1))$. Since
\[6R-\frac{53}{40}(4R+1)=\frac{28R-53}{40}>0\qquad(R\geq2),
\]
the rational inequality implies
\[\frac{c_n^2}{c_{n-1}c_{n+1}}\geq1+\alpha\frac{4R+1}{n(4R-n)}.
\]
Using $(1+x)^\alpha\leq1+\alpha x$ for $x\geq0$ gives
\[\frac{c_n^2}{c_{n-1}c_{n+1}}\geq\left(1+\frac{4R+1}{n(4R-n)}\right)^\alpha=\left(\frac{\binom{4R}{n}^2}
{\binom{4R}{n-1}\binom{4R}{n+1}}\right)^\alpha.
\]
This is the asserted log-concavity.
\end{proof}

\subsection{Completion of the proof}

\begin{proof}[Proof of \cref{thm:dilation-ulc-general}]
The degree formula in \cref{prop:dilation-formula} gives
\begin{align*}
s_{rP}=d+1-\left\lceil\frac{d+1-s_P}{r}\right\rceil=\left\lfloor d+1-\frac{d+1-s_P}{r}\right\rfloor
=\left\lfloor\frac{s_P+(r-1)(d+1)}r\right\rfloor.
\end{align*}
Here we used $\lfloor m-x\rfloor=m-\lceil x\rceil$ for an integer $m$.
If $d=0$, both numerators are $1$, and both assertions follow. We may therefore suppose $d\geq1$. The case $r=1$ is immediate, and
the case $r=2$ is \cref{prop:dilation-double-ulc}.

Suppose $1\leq d\leq4$ and $r\geq3$. Under either hypothesis, $h_P^*(z)$ is log-concave and has no internal zeros, because
$\ULC_1(A)\subseteq\ULC_0(A)$. Moreover, $s_P\leq d<2r$. Zhang's result \cite[Corollary~3.3]{Zhang}, as stated in \cref{thm:zhang-veronese}, therefore shows that $(B_r^{d+1}h_P^*)^{\langle r,0\rangle}=h_{rP}^*$ is real-rooted.
The coefficients of $h_{rP}^*$ are nonnegative and its constant coefficient is one. 
Since its real zeros are negative, it factors into terms of the form $1+\lambda z$ for $\lambda>0$. 
Expanding this factorization reveals that every coefficient up to its degree is positive (where the constant case corresponds to the empty product).
By Newton's inequalities \cite[Theorem~2]{StanleyLC}, we have $h_{rP}^*\in\ULC(s_{rP})$. 
Since $s_{rP}\leq d$, \cref{lem:ulc-order} also gives $h_{rP}^*\in\ULC(d)$.
This proves both assertions in these dimensions.

It remains to consider $d\geq5$ and $r\geq3$. Set $\alpha=53/(40r)$, and choose $A=d$ for the first assertion and
$A=s_P$ for the second. 
In both cases, we have $h_P^* \in \ULC(A)$ by hypothesis. By applying the inclusion established in the previous subsection, we obtain $h_P^* \in \ULC_\alpha(A)$.

Every integer $L\geq6$ has the form $3a+4b$ with $a,b\in\Z_{\geq0}$. 
Indeed, if $L$ is congruent to $0$, $1$, or $2$ modulo $3$, we choose $b=0$, $1$, or $2$, respectively. 
Set $a=(L-4b)/3$. This integer is nonnegative: the smallest admissible
$L$ in the three cases is $6$, $7$, or $8$. 
By applying this fact to $L=d+1$, we get $d+1=3a+4b$.
By \cref{lem:uniform-third-fractional,lem:uniform-fourth-fractional}, setting $R=r-1 \geq 2$, we obtain
\[ 
  B_r^3\in\ULC_\alpha(3(r-1)),\qquad B_r^4\in\ULC_\alpha(4(r-1)).
\]
By repeated application of \cref{lem:fractional-ulc-convolution}, the normalization orders add when the factors are multiplied. 
Thus
\[ B_r^{d+1}=(B_r^3)^a(B_r^4)^b \in\ULC_\alpha\bigl((r-1)(d+1)\bigr).
\]
If $a=0$ or $b=0$, only the factors with positive exponent are used.
By multiplying by $h_P^*$ and applying \cref{lem:fractional-ulc-convolution} once more, we obatin
\[
  B_r^{d+1}h_P^* \in \ULC_\alpha\bigl(A+(r-1)(d+1)\bigr).
\]
Its constant coefficient is $h_0^*(P)=1$, so its zero-offset section is nonzero. 
According to \cref{lem:fractional-ulc-section}, with $e=0$, and \eqref{eq:dilation-extraction}, we have
\[ h_{rP}^*=(B_r^{d+1}h_P^*)^{\langle r,0\rangle} \in\ULC\left(\left\lfloor\frac{A+(r-1)(d+1)}r\right\rfloor\right).
\]
For $A=d$, the numerator inside the floor is $rd+r-1$, and hence the order is $d$. 
For $A=s_P$, the order is $s_{rP}$ by the degree identity at the start of the proof. 
These are the two required conclusions.
\end{proof}

\end{document}